\documentclass{amsart}

\usepackage{amssymb}

\theoremstyle{plain} 

\newtheorem{theorem}{Theorem}[section]
\newtheorem{lemma}[theorem]{Lemma}
\newtheorem{remark}[theorem]{Remark}

\newtheorem{corollary}[theorem]{Corollary}

\newtheorem{proposition}[theorem]{Proposition}
\newtheorem{question}[theorem]{Question}
\newtheorem{hypothesis}[theorem]{Hypothesis}

\newcommand{\makeinvisible}[1]{}

\newcommand{\eop}{\ \hfill $\Box$}

\numberwithin{equation}{section}

\newcommand{\cc}{{\mathbb C}}
\newcommand{\pp}{{\mathbb P}}

\newcommand{\qq}{{\mathbb Q}}
\newcommand{\zz}{{\mathbb Z}}

\newcommand{\Pp}{{\mathcal P}}

\newcommand{\Gg}{{\mathcal G}}

\newcommand{\Ff}{{\mathcal F}}
\newcommand{\Oo}{{\mathcal O}}

\newcommand{\Rr}{{\mathcal R}}

\newcommand{\Jj}{{\mathcal J}}

\newcommand{\End}{{\mathit End}}
\newcommand{\Def}{{\rm Def}}
\newcommand{\Obs}{{\rm Obs}}
\newcommand{\Hilb}{{\rm Hilb}}
\newcommand{\Grass}{{\rm Grass}}

\newcommand{\obs}{{\rm obs}}

\newcommand{\tr}{{\rm tr}}

\newcommand{\Ext}{{\rm Ext}}
\newcommand{\uExt}{\underline{\Ext}}

\newcommand{\rk}{{\rm rk}}

\begin{document}

\author[N. Mestrano]{Nicole Mestrano}
\address{CNRS, Laboratoire J. A. Dieudonn\'e, UMR 6621
\\ Universit\'e de Nice-Sophia Antipolis\\
06108 Nice, Cedex 2, France}
\email{nicole@unice.fr}
\urladdr{http://math.unice.fr/$\sim$nicole/} 

\author[C. Simpson]{Carlos Simpson}
\address{CNRS, Laboratoire J. A. Dieudonn\'e, UMR 6621
\\ Universit\'e de Nice-Sophia Antipolis\\
06108 Nice, Cedex 2, France}
\email{carlos@unice.fr}
\urladdr{http://math.unice.fr/$\sim$carlos/} 
\thanks{This research is partially supported by ANR grant BLAN08-3-352054 (G-FIB)
and by JSPS Grant-in-Aid for Scientific Research (S-19104002)}

%
%

\title[Obstructed bundles]{Obstructed bundles of rank two on a 
quintic surface}

\subjclass[2000]{Primary 14D20; Secondary 14B05, 14J29}

\keywords{Vector bundle, Surface, Moduli space, Deformation, Obstruction}

\begin{abstract}
In this note we consider the moduli space of stable bundles of rank two on
a very general quintic surface. We study the potentially obstructed points of the moduli space
via the spectral covering of a twisted endomorphism.
This analysis leads to generically non-reduced components of the moduli space, and 
components which are generically smooth of more than the expected dimension. We obtain
a sharp bound asked for by O'Grady saying when the moduli space is good. 
\end{abstract}

\dedicatory{To the memory of Masaki Maruyama}

\maketitle

\section{Introduction} \label{sec-introduction}

The moduli space of stable vector bundles on an algebraic surface was
introduced by Maruyama \cite{Maruyama} and Gieseker \cite{Gieseker}.  These
moduli spaces played an important
role in Donaldson's theory as applied to $4$-manifolds which are
complex surfaces. Many authors have closely
studied the structure of the moduli space for big values of $c_2$ 
(for a full set of references see \cite{HuybrechtsLehn}), 
but it remains an interesting and largely open question to understand the structure 
for intermediate values of $c_2$. 

Our objective in this paper is to investigate 
the singularities in the moduli space for the full range of possible values of $c_2$, in one of the first cases of surfaces of general type.

If $E$ is a stable bundle over
a complex algebraic surface $X$, the space
of obstructions to deforming $E$ is $H^2(X,End ^0(E))$ where the superscript $End^0$ denotes the trace-free endomorphisms. 
We say that $E$ is {\em potentially obstructed} if this space is nonzero. 
This definition, for want of better terminology, is not taken to mean that the obstruction 
map is nonzero, it just means that the
dimension of the Zariski tangent space to the moduli space is bigger than the expected
dimension. If the moduli space is good, i.e. generically smooth of the expected dimension
near $E$, then $E$ must be a singular point. The  moduli space might on the
other hand be smooth but overdetermined, i.e. having dimension bigger than the
expected dimension. And of course it could also be overdetermined and singular too.

As is well known (see \cite{Donaldson} \cite{Zuo} \cite[\S 1]{OGradyQuest} \cite{Langer}), 
the dual of the space of obstructions is
$H^0(X,End ^0(E)\otimes K_X)$ by Serre duality. 
An element
$\phi $ in this dual space is a trace-free morphism $\phi : E\rightarrow E\otimes K_X$.
Such a $\phi$ corresponds, by Kuranishi theory,
to an equation of the moduli space locally at $E$, and we call it a {\em co-obstruction}.
A pair $(E,\phi )$ consisting of a bundle together with a nonzero co-obstruction, may
be thought of as a {\em $K_X$-valued Hitchin pair on $X$} \cite{Hitchin}. These pairs are
different from those considered in \cite{hbls} for the surface $X$: the 
Higgs bundles corresponding to representations of $\pi _1$ are endomorphisms taking
values in $\Omega ^1_X$. Over a curve these two notions coincide and indeed Hitchin used
the notation $K_X$ in his original paper \cite{Hitchin}. Generalizing his notation
as written leads to the notion of a Higgs field $E\rightarrow E\otimes K_X$
which is exactly a co-obstruction, often called a ``twisted endomorphism''. 

A basic tool in the analysis of Hitchin pairs is the notion of {\em spectral cover}
\cite{Hitchin} \cite{Donagi} \cite{BNR} \cite{Zuo}.
A twisted endomorphism $\phi :E\rightarrow E\otimes K_X$ gives $E$ the structure of coherent
sheaf on the total space of the vector bundle $K_X$, and the support of the 
coherent sheaf is the {\em spectral covering} associated to $\phi$. 
It consists of the set of pairs $(x,u)$ where $x\in X$ and $u\in K_{X,x}$ such that
$u$ is an eigenvalue of $\phi _x$. 

In our rank two case the spectral cover is
particularly simple to describe: it is the divisor $Z\subset K_X$ determined by the
equation $z^2 = \beta$ where $\beta = \det (\phi )\in H^0(K_X^{\otimes 2})$. 
The {\em Hitchin map} is $(E,\phi )\mapsto \beta$. 

We investigate in a very basic way the possible classification of such spectral covers,
and the implications for the locus of singularities of the moduli space. This follows
Donaldson's original proof of generic smoothness \cite{Donaldson}, as it has been
developed by Zuo in \cite{Zuo}, and more recently by Langer \cite{Langer}. 

Many authors have shown that the moduli spaces of bundles of odd degree on
abelian and K3 surfaces are smooth, going back to \cite{ElencwajgForster} and
Mukai \cite{Mukai}, see the discussions and references in \cite{YoshiokaAbelian}, \cite{YoshiokaK3}. 
O'Grady has observed an important example of symplectic singularities in the moduli of
rank two bundles on a K3 surface \cite{OGradyK3}, along the locus of reducible bundles.
In view of these properties and examples, for understanding bundles on surfaces of
general type it seems like a good idea to look at surfaces of general type which are
as close as possible to K3 surfaces. This motivates our consideration of the example
of a very general quintic surface in $\pp ^3$, where $K_X=\Oo _X(1)$. 
The small degree of $K_X$ in this case will allow us to show that the moduli space 
of rank two bundles of
odd degree is good
for $c_2\geq 10$, and not good for $c_2\leq 9$. 

We have noticed the preprint of Nijsse \cite{Nijsse} who considered also the moduli space
of rank two bundles on a quintic surface. He showed that it is irreducible for $c_2\geq 16$ and
good for $c_2\geq 13$, by adapting and refining O'Grady's method.  By explicit 
arguments we show irreducibility for $c_2\leq 9$, which leaves open the
question of irreducibility 
in the range $10\leq c_2\leq 15$. Other more recent related works are \cite{ChiantiniFaenzi} on ACM bundles on a quintic surface, and \cite{Gulbrandsen} using the Serre construction on abelian threefolds to obtain an explicit description of some component. 

For the present investigation of singularities, the main advantage of the quintic surface  
is that $\beta \in H^0(\Oo _X(2))$ is a quadric
on $X$, and because of its low degree there are not too many possibilities
for the spectral covering. 

If $\beta = 0$ then the Higgs field is nilpotent and
we get a presentation of $E$ as an extension. Furthermore, there are no nontrivial 
line bundles between $\Oo _X$ and $K_X=\Oo _X(1)$ so stability of $E$ yields a particularly
easy description of the extension.  

If $\beta =\alpha ^2$ for $\alpha \in H^0(\Oo _X(1))$
then the spectral cover is a union of two copies of $X$ joined along the curve 
$C$ defined by $\alpha$, the intersection of $X$ with a plane. In this case, $E$ can
be presented as an elementary transformation. Again using the smallness of the 
situation on the quintic surface, we can see that the space of co-obstructions contains
a full copy of ${\mathfrak sl}_2\cong \cc ^3$, in particular it must also contain a
nilpotent element. So we really have a decomposition into two cases $\beta =0$ or
$\beta$ not a square (Proposition \ref{mainclassif}), 
as in \cite{Donaldson} \cite{Zuo} \cite{Langer}. 

If $\beta$ is not a square, it follows from the generality condition on $X$
that $\beta$ defines a reduced curve $D\subset X$ and
$Z$ is irreducible, branched along $D$. In this case, $E$ is the direct image of
a line bundle on the desingularization $\tilde{Z}$ of $Z$, so the dimension of the possible
space of such potentially obstructed bundles is bounded independently of 
$c_2$:
a rough estimate using Lefschetz theory shows the dimension
is $\leq 13$ (Corollary \ref{A13}), a bound which could undoubtedly be improved.
The independence of $c_2$ for this part has been pointed out more generally by Langer \cite{Langer}. 

After this discussion, the main case to be treated is that of nilpotent 
Higgs field. A more precise discussion and classification in this case
occupies the second half of the paper. After considering some examples in \S \ref{sec6},
we consider in detail the singular locus in \S \ref{sec7}. 

Theorem \ref{goodness} says that if $c_2\geq 10$ then the moduli space is good,
by showing that the locus of potentially obstructed
bundles has dimension
strictly smaller than the expected dimension. 

Furthermore, for $c_2\geq 11$
we can identify the biggest irreducible component of the singular locus, 
and show by a general position argument that the quadratic term
of the Kuranishi map has the biggest possible rank at a general point
of this component, in other words the singularity is an ordinary double point
in the transverse direction.

As we shall see, for $c_2\leq 9$ the moduli space is not good,  
so the bound of $10$ is sharp. 
This answers a question of O'Grady
\cite{OGradyBasic} who predicted a sharp bound, 
but in fact our bound improves upon his predicted bound
by $2$. 

An Euler characteristic argument given at the start of \S \ref{sec6} shows that
for $c_2\leq 9$ the general stable bundle has nonzero obstruction space, and indeed 
has a nonzero nilpotent Higgs field $\phi$ as co-obstruction. The second part of \S \ref{sec7}
is devoted to explicit consideration of the possible cases for the scheme $P$ of zeros
of $\phi$. This allows us to show in each case $c_2\leq 9$ that the moduli space is irreducible.

In the case $c_2=9$ there is a
single irreducible component of the moduli space which is completely obstructed but of
the expected dimension, hence generically non-reduced,
with the quadratic term of the Kuranishi map giving an equation of
the form $x^2=0$ at a general point. 

In the case $c_2=8$, Theorem \ref{d8irred} says that the moduli space is irreducible and generically smooth of dimension $13$
whereas the expected dimension is $12$. 

In cases $c_2=6,7$ the moduli space has dimension significantly bigger than the expected
dimension. It is smooth for $c_2=6$ and again completely obstructed, i.e. generically
non-reduced, for $c_2=7$. In cases $c_2=4,5$ the moduli space is relatively simple, see
Lemma \ref{d45}: it is
either an open subset of a $\pp ^1$-bundle over the Grassmanian of lines in $\pp ^3$ (when $c_2=5$), or a $5$-fold covering of the Grassmanian (when $c_2=4$). The moduli space is empty for $c_2\leq 3$. 

The technique used in \S 7 is to notice that the zero-dimensional subscheme $P$ is
on an intersection of several quadrics; we list the possible cases for the intersections
which occur, use the Cayley-Bacharach property, and count dimensions. The case $d=8$ is
perhaps the most subtle.

Our discussion is undoubtedly subsumed by the general theory of the
structure of moduli spaces of stable bundles on surfaces, of Donaldson \cite{Donaldson}, Friedman \cite{Friedman},
Gieseker, Li \cite{GiesekerLi} \cite{Li}, Zuo \cite{Zuo} and others. Many of  these
works have concentrated on the range $c_2\gg 0$, 
although Gieseker \cite{GiesekerCons} and O'Grady \cite{OGradyBasic} construct
components for intermediate values of $c_2$ having more than the expected dimension. 
Langer  \cite{Langer2} gives effective results in characteristic $p$.
We hope that the relatively explicit considerations here can provide some insight
into the complicated middle range where $c_2$ is neither too big nor too small,
continuing in the direction of \cite{Mestrano}. 

Here are some further remarks and questions. 

We obtain some non-reduced components of the moduli space of rank $2$ bundles.
This theory should be somewhat related to the theory of
generically non-reduced components of Hilbert schemes of
curves as constructed following Mumford's original example by 
Kleppe \cite{Kleppe}, Ellia \cite{Ellia}, 
Floystad \cite{Floystad},
Martin-Deschamps and Perrin \cite{MartinDeschampsPerrin}, Azziz \cite{Azziz}, Mukai-Nasu \cite{MukaiNasu} and others. 

The investigation of quintic surfaces here looks 
somewhat similar to the examples discussed in
\cite{ClemensKley}. It would be interesting to understand what happens for a quintic
surface which is no longer very general, i.e. such that the Neron-Severi group has
rank $\geq 2$. 
In a similar spirit, recall 
that smooth quintic surfaces in $\pp ^3$ are connected by deformation to
the ``Horikawa surfaces'' which form a different irreducible component of moduli, the
two intersecting along a locus of quintic surfaces with ordinary double points.
It would be interesting to see how much of our discussion could be done for
Horikawa surfaces. 

Some obvious further questions for the future 
are to see if a similar analysis can be done for 
bundles on a sextic, for bundles of higher rank, and for bundles on less general
quintics. In the context of
our discussion here, for the first values
$c_2=4,5$ the moduli space has considerably more than the expected dimension, which 
seems to indicate the existence of co-obstructions with interesting 
spectral coverings. 
The construction of some irreducible components of the moduli space $M_X(2,-1,d)$
raises the question of studying the Poincar\'e bundles over open or locally
closed subsets of these
components.

We would like to thank Masa-Hiko Saito and Kota Yoshioka for interesting remarks
and questions during the first author's talk at the Workshop on Moduli of Vector Bundles, University of Kobe, July 2009. They motivated us to look more closely at the singular locus of the moduli space of stable bundles. We would also like to thank Andr\'e Hirschowitz and Charles Walter for helpful suggestions.

\section{Obstructions for vector bundles on a surface}

Let $X$ be a smooth projective algebraic surface over $\cc$ 
with ample line bundle $\Oo _X(1)$ 
and corresponding hyperplane class $H:= c_1(\Oo _X(1))\in H^2(X,\qq )$. We consider the
moduli space of $H$-Gieseker-semistable 
vector bundles on $X$ of rank $r$, and given $c_1$ and
$c_2$ denoted $M = M (r,c_1,c_2)$. The open subset of stable points is denoted
$M^s(r,c_1,c_2)$, and the projective moduli space of torsion-free sheaves is denoted
by $\overline{M } (r,c_1,c_2)$. If necessary, the underlying variety $X$ and/or the polarization will be indicated by subscripts as in $M_{X,H}(r,c_1,c_2)$. 

We will concentrate on the case of rank $r=2$ and $c_1=\Oo _X(-1)$
but many of the initial definitions
are valid for any rank. 

Consider a point $E\in M^s(r,c_1,c_2)$. The deformation theory of $E$ is
controlled by the space 
$$
\Def (E):= H^1(\End (E)),
$$
while the obstruction theory is controlled by 
$$
\Obs (E) := H^2(\End ^0(E)).
$$
Here $\End ^0(E):= \ker \left( \tr : \End (E)\rightarrow \Oo _X\right) $
is the trace-free part of the endomorphism bundle of $E$. The map $\tr$ is split
by the diagonal embedding so 
$$
\End (E)=\End ^0(E) \oplus \Oo _X.
$$
The trace-free part is self-dual: $\End ^0(E)^{\ast}\cong \End ^0(E)$
via the pairing
$$
\End ^0(E)\otimes \End ^0(E)\stackrel{\langle \cdot , \cdot \rangle}{\longrightarrow} 
\Oo _X,
$$
with $\langle A , B \rangle := \tr (AB)$.

Let $K_X:= \Omega ^2_X$ denote the dualizing sheaf. By Serre duality,
$$
H^2(\End ^0(E))\cong H^0(\End ^0(E)\otimes K_X)^{\ast}
$$
using $\End ^0(E)\cong \End ^0(E)^{\ast}$. We obtain
$$
\Obs (E)= H^0(\End ^0(E)\otimes K_X)^{\ast},
$$
so $\Obs (E)\neq \{ 0 \}$ if and only if there exists a nonzero element 
$\phi \in \End ^0(E)\otimes K_X$. Such an element may be interpreted as a {\em Higgs field}
\cite{Hitchin} or ``twisted endomorphism''
$$
\phi : E\rightarrow E\otimes K_X
$$
with $\tr (\psi )=0$. Notice that this is a Higgs field twisted by the canonical line bundle
$K_X$ rather than by $\Omega ^1_X$ as in \cite{hbls}. Higgs bundles on higher 
dimensional varieties with twisting by
a line bundle have been considered by a number of authors. 

\begin{lemma}
Suppose $E$ is a vector bundle on a smooth projective surface. Then the obstruction 
space $\Obs (E)$ is nonzero, if and only if there exists a nonzero Higgs field 
$\phi : E\rightarrow E\otimes K_X$
of trace zero. 
\eop
\end{lemma}

A pair $(E,\phi )$ corresponds to a coherent sheaf denoted $\Ff = \Ff _{E,\phi}$ on 
the total space of the line bundle $K_X$, see \cite{BNR} \cite{Donagi} \cite{Hitchin} \cite{Zuo}.
We denote this total space by ${\rm Tot}(K_X)$ or sometimes just $K_X$ by abuse of notation,
with its projection denoted by $p:{\rm Tot}(K_X)\rightarrow X$. 
To construct $\Ff$ note that ${\rm Tot}(K_X)$ is the relative spectrum of the
sheaf of algebras ${\rm Sym}^{\cdot}(K_X^{\ast})$ and $\phi$ corresponds to an action
of this sheaf of algebras on $E$: the coherent sheaf corresponding to the resulting sheaf of
modules is $\Ff$. 
The sheaf $\Ff $ is of pure dimension $2$, and indeed $p_{\ast}(\Ff )=E$
so $\Ff$ should be flat over $X$ if $E$ is to be a vector bundle. 

The {\em spectral surface} is the reduced subscheme $Z\subset {\rm Tot}(K_X)$ 
which is the support of
$\Ff$. It can be viewed as the subset of eigenvalues of $\psi$. Note here that 
we define $Z$ as a reduced subscheme, so $\Ff$ may be a coherent sheaf not on $Z$ but only
on some infinitesimal neighborhood thereof. 

We would like to classify pairs $(E,\phi )$ in the case when $E$ has rank two.
This discussion follows Zuo \cite{Zuo} in a way adapted to our intended application to
quintic hypersurfaces. 

In our case, the {\em Hitchin invariant} of the spectral surface is just the 
determinant $\beta := \det (\phi )\in H^0(X, K_x^{\otimes 2})$. The subscheme $Z$ is defined by the
equation $z^2=\beta$. 

\begin{lemma}
Suppose $\beta \neq 0$. Then $\Ff$ is a rank one torsion-free sheaf on $Z$, flat over $X$. 
\eop
\end{lemma}

\subsection{The obstruction corresponding to $\phi$}

Recall that the local deformation theory of $E$ is 
governed by the  formal {\em Kuranishi map}
$$
\kappa : \widehat{H^1(End (E))}\rightarrow H^2(End^0(E)).
$$
Recall that the Kuranishi map is a
``formal function'' i.e. a power series at the origin represented by a
vector of $h^2(End^0(E))$ 
elements of the complete local ring of the vector space $H^1(End (E))$ at the origin.
This is the reason for the hat in the notation. 

The linear term of $\kappa$ vanishes, and the quadratic term is a function  
$$
\kappa _2 : Sym ^2H^1(End (E)) \rightarrow H^2(End^0(E)) .
$$
For each ``co-obstruction'' or linear form $\phi \in H^2(End^0(E))^{\ast}$ 
we get a function denoted 
$$
\obs (\phi ) = \phi \cdot \kappa _2 : Sym ^2H^1(End (E)) \rightarrow \cc .
$$
By the general theory, we have the formula
\begin{equation}
\label{obsform}
\phi \cdot \kappa _2 (\eta , \eta ' ) =
\obs (\phi ; \eta , \eta ') = \int _X Tr (\phi \cdot [\eta , \eta ']) .
\end{equation}
In this formula $\eta , \eta ' \in H^1(End (E))$ and their commutator is
$$
[\eta , \eta ']\in H^2(End^0(E)),
$$
which is then multiplied by $\phi$ to get an element of $H^2(End(E)\otimes K_X)$.
The trace on the factor $End(E)$ sends us to $H^2(K_X)$ and then we apply
the duality isomorphism $\int _X : H^2(K_X)\stackrel{\cong}{\rightarrow} \cc$.

The easiest way of interpreting \eqref{obsform} is to think of the cohomology classes as being given by their Dolbeault
representatives, and with this formulation the obstruction can be calculated as
$$
obs(\phi ; \eta , \eta ')=
\int _X Tr(\phi \cdot \eta \cdot \eta ' + \phi \cdot \eta' \cdot \eta )
$$
where $\cdot$ indicates matrix multiplication in $End(E)$ coupled with wedge product of
Dolbealt $(0,q)$-forms when necessary. The sign in the commutator is from the
sign relations for differential forms. Using the matrix relation $Tr(AB-BA)=0$
and again the sign conventions,
the obstruction element is equal to 
$$
obs(\phi ; \eta , \eta ')=
\int _X Tr(\phi \cdot \eta \cdot \eta ' - \eta \cdot \phi \cdot \eta ')
= \int _X Tr (Ad(\phi )(\eta ) \cdot \eta ').
$$
Here $Ad(\phi )(\eta ) = \phi \cdot \eta - \eta \cdot \phi$. 
We have 
$$
Ad(\phi )(\eta )\in H^1(End^0(E)\otimes K_X)\cong H^1(End(E))^{\ast}
$$
and the obstruction $obs(\phi ; \eta , \eta ')$ is just the duality pairing between
$Ad(\phi )(\eta )$ and $\eta '$. 

\begin{proposition}
\label{rankeq}
The rank of the quadratic form on $H^1(End^0(E))$ given by the obstruction $obs(\phi )$,
is equal to the rank of the linear  map 
$$
Ad(\phi ) : H^1(End (E))\rightarrow H^1(End (E)\otimes K_X).
$$
\end{proposition}
\begin{proof}
The duality pairing is a perfect pairing so the rank of the quadratic form is
equal to the rank of the linear map. 
\end{proof}

\section{The case of nilpotent co-obstruction}
\label{sec-nilpotent}

In this section we study the case when $\phi : E\rightarrow E\otimes K_X$ is
a nonzero nilpotent matrix at the general point of $X$. In this case, $\ker (\phi )$
is a saturated subsheaf of rank $1$ in $E$, thus it is an invertible sheaf which we denote
$L$. The quotient $E/L$ is a torsion-free sheaf of rank $1$
whose double dual is a line bundle which
we denote by $L'$. Put $\Jj _P:= (L')^{\ast}\otimes (E/L)$. It is the ideal in $\Oo _X$ 
of a zero-dimensional subscheme which we denote by $P$. The bundle $E$ fits into
an exact sequence 
\begin{equation}
\label{ljl}
0\rightarrow L \rightarrow E \rightarrow \Jj _P\otimes L' \rightarrow 0
\end{equation}
where the nilpotent Higgs field $\phi$ factors as 
$$
E\rightarrow \Jj _P\otimes L'\rightarrow L\otimes K_X \rightarrow E\otimes K_X.
$$
The use of extensions such as \eqref{ljl} goes back to Serre's construction, and has continued systematically ever since \cite{ElencwajgForster} \cite{OGradyBasic} \cite{Nijsse} 
\ldots .

In order to insure the existence of a locally free extension of this form, one 
introduces the following condition: if $U$ is a line bundle then 
say that $P$ {\em satisfies the Cayley-Bacharach property for sections
of $U$} if for any subscheme $P'\subset P$ with $\ell (P')=\ell (P)-1$,
the conditions imposed by $P$ and $P'$ on sections of $U$ are the same, in other
words
$$
\ker \left( H^0(U)\rightarrow H^0(P,U|_P) \right)  = 
\ker \left( H^0(U)\rightarrow H^0(P',U|_{P'}) \right) .
$$

\begin{lemma}
\label{totallysuper}
Given line bundles $L$ and $L'$ and a zero-dimensional subscheme $P$ of length $d$, 
there exists a rank $2$
vector bundle fitting into an extension \eqref{ljl}, if and only if $P$ is
a local complete intersection
satisfying the Cayley-Bacharach property  
for sections of $L'\otimes L^{\ast} \otimes K_X$. 
Let $c$ be the number of
conditions imposed by $P$ on $H^0(L'\otimes L^{\ast} \otimes K_X)$, 
and suppose $h^1(L'\otimes L^{\ast}\otimes K_X)=0$, then
$$
\dim Ext ^1(\Jj _P\otimes L', L) = d-c.
$$
\end{lemma}
\begin{proof}
Here is a brief account of this well-known fact (which was used implicitly in
\cite[(3.29)]{OGradyBasic} for example, see also \cite{Nijsse}). 
A locally free extension
exists locally if and only if $P$ is a local complete intersection. 
Set $U:= L'\otimes L^{\ast} \otimes K_X$ and $U|_P=U\otimes \Oo _P$. 
Consider the
long exact sequence
$$
0\rightarrow H^0(U\otimes \Jj _P)\rightarrow H^0(U)
\stackrel{\epsilon}{\rightarrow} H^0(U|_P)
\rightarrow H^1(U\otimes \Jj _P)\rightarrow H^1(U)\rightarrow 0.
$$
By duality $Ext ^1(\Jj _P\otimes L', L)=H^1(U\otimes \Jj _P)^{\ast}$, so an
extension restricts to
a linear form $f$ on $H^0(U|_P)$
which vanishes on the image of $H^0(U)$. It is quite classical 
that the condition for this to yield a locally free extension, is that $f(I)\neq 0$
for any subsheaf $I\subset U|_P$ which can
be supposed of rank one. Such a subsheaf corresponds to the ideal of a subscheme
$P'\subset P$ of colength $1$. Saying that the general form $f$ which vanishes on
the image of $H^0(U)$, is nonzero on $I$, is equivalent to
saying that $I$ is not in the image of $H^0(U )$, which in turn says that 
$P'$ imposes the same number of conditions as $P$. 

The number $c$ is the rank of the restriction map $\epsilon$, and $h^0(U|_P)=d$ is
the length of $P$, so assuming that $h^1(U)=0$ the dimension of 
$Ext ^1(\Jj _P\otimes L', L)$ is $d-c$. 
\end{proof}

Given a vector bundle $E$ sitting in a sequence of the form \eqref{ljl},
we have $\det (E)\cong L\otimes L'$, 
$$
E^{\ast} \otimes L' \cong E \otimes L^{\ast},
$$
$$
End ^0(E)\otimes L \otimes L' \cong Sym ^2(E) ,
$$
and there is an exact sequence
\begin{equation}
\label{elsub}
0\rightarrow E \otimes (L')^{\ast}
\rightarrow End ^0(E)\rightarrow \Jj _P^2\otimes L'\otimes L^{\ast} \rightarrow 0.
\end{equation}

\begin{lemma}
\label{gdef}
Taking the dual of \eqref{elsub} gives an exact sequence of the form
$$
0\rightarrow L\otimes (L')^{\ast} \rightarrow End ^0(E)
\rightarrow \Gg \rightarrow 0
$$
where $\Gg$ fits into an exact sequence of the form
$$
0\rightarrow \Gg \rightarrow E^{\ast}\otimes L' \rightarrow \uExt ^2(\Oo _{2P},\Oo _X)\otimes
 L\otimes (L')^{\ast}\rightarrow 0,
$$
and $2P$ denotes the subscheme defined by $\Jj _P^2$. 
\end{lemma}
\begin{proof}
Taking the dual gives a long exact sequence of the form
$$
\ldots \rightarrow End ^0(E)\rightarrow E^{\ast}\otimes L'
\rightarrow \uExt ^1(\Jj _P^2\otimes L'\otimes L^{\ast} , \Oo )\rightarrow 0.
$$
But the long exact sequence for the standard sequence defining $\Oo _{2P}$ gives
$$
\uExt ^1(\Jj _P^2,\Oo _X)\cong \uExt ^2(\Oo _{2P},\Oo _X).
$$
\end{proof}

The nilpotent Higgs field $\phi$ factors as 
$$
E\rightarrow \Jj _P\otimes L'\rightarrow L\otimes K_X \rightarrow E\otimes K_X
$$
where the middle map comes from a map of line bundles
denoted 
$$
\zeta : L'\rightarrow L\otimes K_X.
$$
with transpose $\zeta ^t : L^{\ast}\rightarrow (L')^{\ast}\otimes K_X$.

\begin{proposition}
\label{bigdiag}
In the situation of an exact sequence \eqref{ljl} with $\Gg$ defined as in 
Lemma \ref{gdef}, the map 
$$
Ad(\phi ): H^1(End ^0(E))\rightarrow H^1(End ^0(E)\otimes K_X)
$$
factors as the map fitting into the following diagram:
$$
\begin{array}{ccc}
& {\scriptstyle H^0((L'\otimes L^{\ast}\otimes K_X)|_{2P})^{\ast}} & \\
& \downarrow & \\
\cdots \;\; 
H^1(End ^0(E)) & \longrightarrow \;\; H^1(\Gg )\;\;  \longrightarrow & H^2(L\otimes (L')^{\ast})   \;\; \cdots \\
& \downarrow & \\
& H^1(E\otimes L^{\ast}) \;\; \stackrel{\zeta ^t}{\longrightarrow}  &H^1(E\otimes (L')^{\ast}\otimes K_X) \\
& \downarrow &\downarrow \\
& 0 &H^1(End ^0(E)\otimes K_X) 
\end{array}
$$
where the main horizontal and vertical sequences are exact. 
\end{proposition}
\begin{proof}
The exact sequences are just the long exact sequences associated to the
exact sequences of Lemma \ref{gdef}. 
The element at the top is explained by
$$
H^0(\uExt ^2(\Oo _{2P},\Oo _X)\otimes
L\otimes (L')^{\ast}) = Ext ^2(\Oo _{2P}\otimes L^{\ast} \otimes L'\otimes K_X, K_X)
$$
$$
\cong H^0(\Oo _{2P}\otimes L^{\ast} \otimes L'\otimes K_X)^{\ast}.
$$
The factorization is obtained by factoring the map $Ad(\phi )$ on the level of
sheaves, then applying $H^1$. 
\end{proof}

\begin{remark}
\label{whatmap}
The composed map
$$
H^0(2P, \Oo _{2P}\otimes L'\otimes L^{\ast}\otimes K_X)^{\ast}
\rightarrow 
H^2(L\otimes (L')^{\ast})
$$
is dual to a map
$$
H^0(L'\otimes L^{\ast}\otimes K_X)\rightarrow 
H^0(2P, \Oo _{2P}\otimes L'\otimes L^{\ast}\otimes K_X),
$$
which is just the evaluation map for sections over the subscheme $2P$ up to multiplying
by a unit. 
\end{remark}

It will later be useful to have a name for the moduli variety of bundles with
nilpotent co-obstruction,
which has appeared in \cite{OGradyBasic} \cite{Nijsse} \cite{Zuo} \cite{Langer} 
for example.
Let $\Sigma _d(X; L)$ denote the moduli variety of extensions  \eqref{ljl},
in other words the moduli variety of pairs $(E,\eta )$ where $E$ is a rank two bundle
(with fixed determinant which we leave out of the notation),
and $\eta : L\rightarrow E$ is a morphism whose cokernel is torsion-free of colength $d$ in
its double dual. The map which to $(E,\eta )$ associates the subscheme defined by 
$({\rm coker}(\eta )^{\ast \ast}/{\rm coker}(\eta )$ gives a map $\Sigma _d(X; L)\rightarrow
\Hilb ^{\rm lci}_d(X)$ to the subset of the Hilbert scheme parametrizing zero-dimensional
local complete intersections. 

Recall that the {\em dualizing sheaf} of $P$ is 
\begin{equation}
\label{wP}
\omega _P := \uExt ^2(\Oo _{P},K_X) = \uExt ^1(\Jj _P, K_X).
\end{equation}
If $P$ is a local complete intersection then noncanonically $\omega _P\cong \Oo _P$.

We have the following explicit description of $\Sigma _d(X; L)$: it is
the variety of pairs $(P, \xi )$ where $P\in \Hilb ^{\rm lci}_d(X)$
and $\xi$ is a nonzero map up to scalars, composing to zero
in the sequence
\begin{equation}
\label{xiseq}
\cc \stackrel{\xi}{\rightarrow} H^0(\omega _P\otimes L'\otimes L^{\ast})
\stackrel{\epsilon ^{\ast}}{\rightarrow} 
H^0(L'\otimes L^{\ast} \otimes K_X)^{\ast},
\end{equation}
such that $\xi _z\neq 0$ for every closed point $z$ of $P$. This is the condition
which occurred in Lemma \ref{totallysuper}.  

\begin{corollary}
\label{dimsigmageq}
Every irreducible component of $\Sigma _d(X; L)$ has dimension 
$\geq 3d-h^0(L'\otimes L^{\ast} \otimes K_X)-1$.
\end{corollary}
\begin{proof}
The open subset  $\Hilb _d^{\rm lci}(X)\Hilb _d(X)$ consisting of 
local complete intersection subschemes, is open. This can be seen by 
using semicontinuity of the dimension of ${\rm Tor}^1(\Oo _P,\Oo _{\{ z\}})$
for closed points $z\in |P|$; an infinitesimal version is given in 
\cite{Vistoli}.  Over this open subset, we have a universal subscheme $\Pp \subset 
X\times \Hilb _d^{\rm lci}(X)$ and we get a bundle denoted 
$p_{2,\ast}\omega _{\Pp}(-2)$. The total space of this bundle, minus the subspaces
consisting of sections $\xi$ vanishing at some closed points, is a smooth 
variety of dimension $3d$. The subscheme of $(P,\xi  )$ such that 
$\epsilon ^{\ast}(\xi  )=0$, is defined by $h^0(L'\otimes L^{\ast} \otimes K_X)$ equations so all irreducible
components have dimension $\geq 3d-h^0(L'\otimes L^{\ast} \otimes K_X)$. When we divide out by scalar multiplication
on $\xi$ we get a scheme all of whose components have dimension $\geq 3d-
h^0(L'\otimes L^{\ast} \otimes K_X) -1$.
\end{proof}

\section{The case of reducible spectral surface}
\label{sec-reducible}

In this section we study the special case when the spectral surface
decomposes into two irreducible components meeting along a smooth curve.

\begin{hypothesis}
\label{squarehyp}
Suppose that $E$ is a slope-stable bundle with co-obstruction $\phi$ such 
that $\det ( \phi ) =\alpha ^2$ for
a nonzero 
section $\alpha \in H^0(X,K_X)$. Assume that the curve $C\subset X$
defined by $\alpha =0$ is reduced.
\end{hypothesis}

In this case, $Z = Z^+\cup Z^-$ where $Z^+\subset {\rm Tot}(K_X)$ is the graph of $\alpha$
in the total space of the line bundle $K_X$, 
and $Z^-$ is the graph of $-\alpha$. Note that $Z^{\pm}\cong X$ and $Z^+\cap Z^- =C$ the latter
being contained in the zero-section of ${\rm Tot}(K_X)$. 

Elementary transformations are a classical method for constructing bundles
introduced by Maruyama \cite{MaruyamaET}
\cite{MaruyamaTransform}, and recently for example
Nakashima constructs stable vector bundles on CY threefolds using this technique
\cite{Nakashima}.
The following proposition says that $E$ is obtained by an elementary transformation along
$C$, and determines the structure of $\phi$.

\begin{proposition}
\label{squarecase}
Assume Hypothesis \ref{squarehyp}. With the previous notations, 
the restriction $\Ff |_{Z^+}$ determines a line bundle denoted $L$ on $Z^+\cong X$, and the restriction
$\Ff |_{Z^-}$ determines a line bundle $L'$ on $Z^-\cong X$. We have an exact sequence 
\begin{equation}
\label{mainseq}
0\rightarrow E \rightarrow L\oplus L' \rightarrow i_{\ast}R \rightarrow 0
\end{equation}
where $R$ is a rank one torsion-free sheaf over $C$ and $i:C\hookrightarrow X$ denotes
the inclusion. The $K_X$-valued endomorphism $\phi$ is 
the operator on $E$ induced by the $K_X$-valued endomorphism of $L\oplus L'$
whose matrix is $\left( 
\begin{array}{cc}
\alpha & 0 \\ 0 & -\alpha \end{array} \right) $. 

Conversely, any pair of line bundles $L,L'$ and such a surjection $L\oplus L' 
\twoheadrightarrow i_{\ast}R$ determines a pair
$(E,\phi )$. 
\end{proposition}
\begin{proof}
See \cite{Zuo}.
The sheaf $\Ff$ on the total space ${\rm Tot}(K_X)$ corresponding to $(E,\phi )$ \cite{BNR}
\cite{Donagi} \cite{Zuo} 
is torsion-free and saturated in the sense that it admits no extension 
which is an isomorphism outside codimension $2$. Take the restriction $\Ff |_{Z^{+}}$,
take the quotient 
by torsion, and look at the double dual. This is a line bundle $L$ and we have
a map $E\rightarrow L$. Similarly for $\Ff |_{Z^{-}}$ we get a map to a line bundle
$E\rightarrow L'$. This gives an injection of sheaves $E\hookrightarrow L\oplus L'$,
an isomorphism away from $C$.
The quotient $\Rr $ has to be nonzero, indeed $E$ is assumed to be slope-stable so it
is indecomposable. But, the quotient must also be pure of dimension $1$,
since $E$ is reflexive. Along the smooth locus of $C$ (which is dense because
we are assuming that $C$ is reduced), the quotient
has to be of rank one---if it were of rank $2$ then $E$ would be of the form
$(L\oplus L')(-C)$, again decomposable.  We get the structure result
of the proposition, away from the finite singular set of $C$. In particular,
$\Rr $ is supported scheme-theoretically on $C$ at the smooth locus, but this then
has to hold at the singular points too because $\Rr $ has pure dimension $1$.
We get $\Rr =i_{\ast}(R)$ for a torsion-free sheaf $R$ of rank $1$ on $C$.  
The statement identifying $\phi$ holds outside the singular set of $C$,
and then extends. 
\end{proof}

In terms of the description of the previous proposition, we have 
$$
c_1(E) = L + L' - C,
$$
$$
c_2(E) = c_1(L)c_1(L') - (L.C + L'.C) +{\rm deg}_C(R) . 
$$
On any surface $X$ which admits a reduced canonical divisor, 
the above proposition gives a different construction of 
potentially obstructed stable bundles $E$ of arbitrarily high $c_2$,
by taking $R$ to have very high degree.

We next consider the obstruction determined by $\phi$. 
This discussion is just a start and is not used elsewhere; it is included
here for eventual future reference. In what follows, we assume that the
ramification curve $C$ is smooth. 

The first step is to calculate explicitly the deformation space $H^1(End (E))$.
From \eqref{mainseq} we get maps $L(-C)\rightarrow E$ and the same for $L'$,
which gives an exact sequence
\begin{equation}
\label{secondseq}
0\rightarrow (L\oplus L')(-C)\rightarrow E\rightarrow S\rightarrow 0.
\end{equation}
Here $S$ is again a line bundle supported on $C$, and we have two exact sequences,
\begin{equation}
\label{CseqL}
0\rightarrow S\rightarrow (L\oplus L')|_C \rightarrow R\rightarrow 0
\end{equation}
and
\begin{equation}
\label{CseqE}
0\rightarrow R(-C)\rightarrow E|_C\rightarrow S\rightarrow 0.
\end{equation}
The dual of \eqref{secondseq} is
\begin{equation}
\label{dualseq}
0\rightarrow E^{\ast}\rightarrow (L\oplus L')^{\ast}(C)\rightarrow S^{\ast}(C)\rightarrow 0.
\end{equation}
Using the left arrows in \eqref{mainseq} and \eqref{secondseq} we get an injection 
of sheaves; and define $Q$ to be the quotient, to give altogether an exact sequence
\begin{equation}
\label{endseq}
0\rightarrow End(E) \stackrel{u}{\rightarrow}
End(L\oplus L')(C) \rightarrow Q \rightarrow 0.
\end{equation}
This gives a long exact sequence of cohomology
\begin{equation}
\label{leseq}
\ldots 
H^0 (End(L\oplus L')(C) )\rightarrow 
H^0(Q)\stackrel{\delta}{\rightarrow} H^1(End(E) )\rightarrow 
H^1 (End(L\oplus L')(C))\rightarrow \ldots .
\end{equation}
The main piece of $H^1(End(E) )$ will come from the
connecting map $\delta$, so we would like to understand the structure of $Q$. 
The right maps of \eqref{mainseq} and \eqref{dualseq} give 
a map fitting into an exact sequence defining a sheaf $G$,
$$
0\rightarrow G(-C) \rightarrow 
End(L\oplus L')\rightarrow 
(L\oplus L')^{\ast}\otimes R\,  \oplus \, 
S^{\ast}\otimes (L\oplus L')\rightarrow 
S^{\ast}\otimes R \rightarrow 0.
$$
This leads to an exact sequence 
\begin{equation}
\label{qcseq}
0\rightarrow Q|_C\rightarrow (L\oplus L')^{\ast}\otimes R(C)\,\,  \oplus \, \,
S^{\ast}\otimes (L\oplus L')(C)\rightarrow 
S^{\ast}\otimes R(C) \rightarrow 0.
\end{equation}
However, $Q$ is not supported on $C$, but only on the second infinitesimal neighborhood.
This can be seen by looking at the contribution of $Q|_C$ to $c_1(End(E))$,
one can see that we are missing a piece of rank one supported on $C$. There is 
an exact sequence
$$
0\rightarrow T \rightarrow Q \rightarrow Q|_C\rightarrow 0
$$
where $T$ is the image of the map $Q(-C)\stackrel{m}{\rightarrow} Q$. 
Note that $T$ is a line bundle supported on $C$. 

On the infinitesimal neighborhood $2C$ the sheaf
$W= \uExt ^1(Q,\Oo _X)$ fits as the kernel in the exact sequence
\begin{equation}
\label{wseq}
0\rightarrow W \rightarrow End (L\oplus L')(C)|_{2C} \rightarrow Q \rightarrow 0.
\end{equation}

The above discussion passes to trace-free parts, denoted by a supercript $(\; )^0$. For example $W$ and $Q$ split as
$$
W\cong W^0\oplus \Oo _C,\;\;\; Q \cong Q^0\oplus \Oo _C(C)
$$
and these are compatible with the splitting of the exact sequence \eqref{wseq}
into 
$$
0\rightarrow \Oo _C \rightarrow \Oo _{2C}(C)\rightarrow \Oo _C(C)\rightarrow 0
$$
direct sum with
\begin{equation}
\label{woseq}
0\rightarrow W^0 \rightarrow End ^0(L\oplus L')(C)|_{2C} \rightarrow Q^0 \rightarrow 0.
\end{equation}
Furthermore, we have the exact sequences 
\begin{equation}
\label{wbigseq}
0\rightarrow End (L\oplus L')(-C) \rightarrow End (E) \rightarrow W\rightarrow 0
\end{equation}
and 
\begin{equation}
\label{wobigseq}
0\rightarrow End ^0(L\oplus L')(-C) \rightarrow End ^0(E) \rightarrow W^0\rightarrow 0
\end{equation}

\begin{lemma}
\label{defident}
We have a perfect pairing 
$$
H^0(Q^0)\otimes H^0(Q^0\otimes K_X) \rightarrow \cc
$$
fitting with isomorphisms
$$
H^0(Q^0\otimes K_X)\cong H^1(W^0) \cong H^0(Q^0)^{\ast},
$$
$$
H^0(Q^0)\cong H^1(W^0\otimes K_X ) \cong H^0(Q^0\otimes K_X )^{\ast}.
$$
The connecting map for the sequence \eqref{woseq} is the map
\begin{equation}
\label{woseqconn}
H^0(Q^0)\rightarrow H^1(W^0)
\end{equation}
obtained by the composition of the connecting map for the trace-free version of 
\eqref{endseq} with the right map of \eqref{wobigseq},
$$
H^0(Q^0)\rightarrow H^1(End^0(E))\rightarrow H^1(W^0) .
$$
If, furthermore we assume that $H^1(End^0(L\oplus L')(C))= 0$ and 
$H^1(End^0(L\oplus L ')(-C))= 0$ then $H^1(End^0(E))$
is the image of \eqref{woseqconn}.
\eop
\end{lemma}

Apply Proposition \ref{rankeq} to the the present situation.

\begin{corollary}
Suppose $H^1(End^0(L\oplus L')(C))= 0$ and 
$H^1(End^0(L\oplus L ')(-C))= 0$. Then the rank of the 
quadratic form corresponding to $\obs (\phi )$ is the same as the rank of the
composed linear map 
$$
H^0(Q^0)\stackrel{Ad(\phi )}{\rightarrow} H^0(Q^0\otimes K_X)
\rightarrow H^1(W^0\otimes K_X) .
$$
\end{corollary}
\begin{proof}
By Lemma \ref{defident}, $H^1(End^0(E))$ is the image of the map 
$H^0(Q^0)\rightarrow H^1(W^0)$. Similarly, 
$H^1(End^0(E)\otimes K_X)$ is the image of the map 
$H^0(Q^0\otimes K_X)\rightarrow H^1(W^0\otimes K_X)$.
We have a commutative diagram 
$$
\begin{array}{ccccc}
H^0(Q^0) & \rightarrow & H^1(End^0(E)) & \rightarrow & H^1(W^0) \\
\downarrow & & \downarrow && \downarrow \\
H^0(Q^0\otimes K_X) & \rightarrow & H^1(End^0(E)\otimes K_X) & \rightarrow & H^1(W^0\otimes K_X)
\end{array}
$$
where the vertical maps are given by $Ad(\phi )$. These compatibilities can
be seen from the fact that $Ad(\phi )$ acts on the exact sequences of the form
\eqref{mainseq}, \eqref{wbigseq} \ldots and the conclusion comes from a diagram chase.
\end{proof}

Further work would be needed to obtain a full calculation of $Ad(\phi )$.

\section{On a very general quintic surface}

From now on we specialize to the case when $X\subset \pp ^3$ is a very general
quintic surface. By ``very general'' we mean smooth and at least that 
$$
Pic (X)\cong Pic(\pp ^3) = \zz .
$$
This condition holds on the complement of the {\em Noether-Lefschetz locus} which is
a countable union of subvarieties. Further genericity conditions will be added 
as necessary, particularly
in \S \ref{sec7}. 

Notice that
$$
H^1(\Oo _X)=0\, ,\;\; H^2(\Oo _X) = \cc ^4  \, , \;\; K_X = \Oo _X(1)\, ,
$$
and if $n\in \zz$ then  
$$
H^1(\Oo _X(n))=0.
$$
This is seen by looking at the piece of long exact sequence 
$$
H^1(\pp ^3, \Oo _{\pp ^3}(n)) \rightarrow H^1(\Oo _X(n))
\rightarrow H^2(\pp ^3, \Oo _{\pp ^3}(n-5))
$$
and using $H^i(\pp ^3, \Oo _{\pp ^3}(m))=0$ for $i=1,2$ and all $m$.
Similarly, for $m=0,1,2$ we have
$$
H^0(\Oo _X(m))= H^0(\Oo _{\pp ^3}(m)).
$$

Since $Pic(X)\cong \zz$ all choices of hyperplane class give the same notion
of stability and we don't need to include this choice in the notation. 
Also, for bundles of degree $c_1(E)=-1$ the four notions of 
Gieseker, slope / stability, semistability coincide.

The expected dimension of the moduli space $M = M _X(2,-1, d)$ of stable bundles
of rank $\rk (E)=2$, with $c_1(E)=\Oo _X(-1)$ and $c_2(X)=d$, is 
\begin{equation}
\label{ed}
\dim ^{\rm exp}_E(M ) = 4d -20.
\end{equation}

\subsection{First classification of potentially obstructed bundles}

Recall that $E$ is said to be 
{\em potentially obstructed} if $h^2(End^0(E))>0$. This is equivalent to
saying that $E$ is either a singular point of the moduli space, or lies in an
irreducible component whose dimension is strictly greater than the expected dimension (or both). The co-obstruction may be assumed to be either nilpotent or have an irreducible spectral cover \cite{Zuo} \cite{Langer}: 

\begin{proposition}
\label{mainclassif}
Suppose $X\subset \pp ^3$ is a general quintic surface. Suppose $E$ is a
stable bundle with $det (E)\cong \Oo _X(-1)$, and $E$ is potentially obstructed. Then either:
\newline
(i)---there exists an exact sequence 
\begin{equation}
\label{iseq}
0\rightarrow \Oo _X(-1)\rightarrow E \rightarrow \Jj _{P} \rightarrow 0
\end{equation}
where $\Jj _P\subset \Oo _X$ is the ideal of a zero-dimensional subscheme $P\subset X$;
or else 
\newline
(ii)---there is a section $\beta \in H^0(\Oo _X(2))$ which is not a square, defining
a double cover $r:Z\rightarrow X$ with $Z\subset K_X$ and $r$ ramified along $Zero (\beta )$,
together with a line bundle $L$ over a desingularization $\varepsilon : \tilde{Z}\rightarrow Z$
such that $E\cong r_{\ast}\varepsilon _{\ast}(L)^{\ast \ast}$. 
\end{proposition}
\begin{proof}
There is a nonzero twisted endomorphism $\phi : E\rightarrow E(1)$.
Let $\beta := \det (\phi )\in H^0(\Oo _X(2))$. If $\beta = 0$ then $\phi$ is 
nilpotent at the general point, so we are in case (i) by the discussion of
\S \ref{sec-nilpotent}. 

Assume from now on $\beta \neq 0$.

Let $D:= Zero (\beta )$,
which is a divisor in the linearly equivalence class corresponding  to $\Oo _X(2)$.
Since $Pic (X)=\zz$ with generator $\Oo _X(1)$, either $D$ is irreducible and reduced,
or else $D=2C$ with $C$ in the linear system of $\Oo _X(1)$. 

If $D$ is irreducible and reduced, we are in case (ii). 

The remaining case is when $D=2C$. Then, 
using the fact that the
sections of $\Oo _X(1)$ all come from $\pp ^3$ and that the restriction
on quadrics is injective, we get $\beta = \alpha ^2$ for $\alpha$ a linear form. 
Notice that $C$ itself has to be irreducible and reduced since it is indecomposable in
the positive cone of the Neron-Severi group, so 
the theory of \S \ref{sec-reducible} applies.
Notice that Hypothesis \ref{squarehyp} holds because $C$ is reduced. 
Proposition \ref{squarecase} expresses $E$ as an elementary
transformation
$$
0\rightarrow E \rightarrow L\oplus L' \rightarrow i_{\ast}(R)\rightarrow 0
$$
with $L,L'$ line bundles and $R$ a torsion free rank one sheaf on $C$. 
The sheaf $i_{\ast}(R)$ is supported on $C$, so $L(-C)$ is included in $E$
and stability says $\deg (L)-1 < -1/2$ hence $\deg (L)\leq 0$. Similarly
for $L'$, but $\deg (L)+\deg (L')=0$ so  $L\cong L'\cong \Oo _X$. 
Our elementary transformation can thus be written
$$
0\rightarrow E \rightarrow \Oo _X\oplus \Oo _X \rightarrow i_{\ast}(R)\rightarrow 0.
$$
But since $i_{\ast}(R)$ is supported on $C$ we get 
$$
\Oo _X(-1)\oplus \Oo _X(-1) \rightarrow E.
$$
Either one of the maps $\Oo _X(-1)\rightarrow E$ leads to an 
exact sequence of the form \eqref{iseq}. 
\end{proof}

\noindent
{\em Remark}---
In the last case of the proof,
for any $2\times 2$ matrix $m$ with complex entries we get a map
$$
E\rightarrow \Oo _X\oplus \Oo _X \stackrel{m}{\rightarrow} (\Oo _X(-1)\oplus \Oo _X(-1) )(1) \rightarrow E(1),
$$
so on trace-free parts this gives a whole
$$
{\mathfrak sl}_2(\cc )= End ^0(\cc ^2) \subset H^0(End ^0(E)\otimes K_X).
$$
In other words, the co-obstruction $\phi$ is not unique but lies in a space of
dimension at least $3$. The image of a nilpotent element of  
${\mathfrak sl}_2(\cc )$ is a new co-obstruction $\psi$ which is nilpotent.

\subsection{Case (ii)}

In this case we can give a bound for the dimension of the potentially obstructed
locus, independent of $c_2$. This has been pointed out for general surfaces by
Langer \cite[proof of Theorem 6.3.2]{Langer}, but a more precise bound can be given
in our particular case. 

Let $A$ be the maximum irregularity of a surface $\tilde{Z}$ arising as the
desingularization of a spectral cover $Z$ associated to a non-square $\beta \in H^0(\Oo _X(2))$. Notice that the family of possible double covers $r:Z\rightarrow X$ arising in Case (ii) is bounded,
i.e. parametrized by a single finite-type constructible set. Hence the same can be said
for the possible desingularizations $\tilde{Z}$, and we get a global bound 
$q(\tilde{Z})\leq A$. Recall that $\dim Pic ^0(\tilde{Z}) = q(\tilde{Z})$.

\begin{proposition}
\label{iibound}
The dimensions of components of the spaces of potentially obstructed 
bundles falling into case (ii)
of Proposition \ref{mainclassif}, are bounded by a constant $A+8$ which is independent of $d=c_2(E)$.
\end{proposition}
\begin{proof}
The dimension of the space of $\beta$ is bounded by $9$, and for each choice of $\beta$
the dimension of the space of rank one bundles on $Z$ is bounded by $A$. However,
at a general value of $\beta$ the spectral cover $Z$ is smooth itself and has 
irregularity $0$, so the estimate of $A+9$ can be reduced to $A+8$.
This kind of argument in the style of Ngo's dimension estimates
for the Hitchin fibers in the case of curves \cite{Ngo}, can probably be pushed further. 
\end{proof}

Using Lefschetz theory we can give a bound $A\leq 5$. Let $D\subset Z$ denote
the ramification locus of $r$ i.e. $D=Zero (\beta )$, and let $B\subset \pp ^3$
be a general plane. For this discussion, denote by an asterisk $(\; )^{\ast}$ the complement
of the intersection of the variety in question with $D$ or its preimage.

The Lefschetz theorem for the quasiprojective variety $X^{\ast} = X-D$ 
gives a surjection
$$
\pi _1(X^{\ast}\cap B)\twoheadrightarrow \pi _1(X^{\ast}).
$$
Let $Z_B$ denote the preimage in $Z$ of $X\cap B$, 
which  is the same as
the preimage in $\tilde{Z}$; thus $Z_B^{\ast}= Z_B^{\ast} - r^{-1}(D\cap B)$.
Note that $\tilde{Z}^{\ast}=Z^{\ast}$ is an etale covering of $X^{\ast}$
with group $\zz /2$
so there is a diagram 
$$
\begin{array}{ccccc}
\pi _1(Z_B^{\ast}) & \hookrightarrow & \pi _1(X^{\ast}\cap B) &\rightarrow & \zz /2  \\
\downarrow & & \downarrow & & \downarrow {\scriptstyle = }\\
\pi _1(\tilde{Z}^{\ast}) & \hookrightarrow & \pi _1(X^{\ast})&\rightarrow & \zz /2
\end{array}
$$
where the middle vertical arrow is surjective. It follows that the left vertical
arrow is surjective too (this result is a somewhat more advanced statement
but which is still standard in Lefschetz theory). On the other hand the
diagram 
$$
\begin{array}{ccc}
\pi _1(Z_B^{\ast}) & \rightarrow & \pi _1(Z_B) \\
\downarrow & & \downarrow \\
\pi _1(\tilde{Z}^{\ast}) & \rightarrow & \pi _1(\tilde{Z})
\end{array}
$$
with surjective map on the bottom,
shows that the map $\pi _1(Z_B) \rightarrow \pi _1(\tilde{Z})$ is surjective.
It follows that the restriction map on Picard groups is injective:
\begin{equation}
\label{picrestrict}
Pic^0(\tilde{Z})\hookrightarrow Pic^0(Z_B).
\end{equation}
However, the line bundles on $\tilde{Z}$ are antipreserved by the involution 
(which acts, for example because we can choose an equivariant resolution $\tilde{Z}$
\cite{Bierstone}), because $Pic^0(X)=0$. Thus, the image of the map
\eqref{picrestrict} lies in the Prym variety $Prym(Z_B/X)$. This proves the
following bound. 

\begin{lemma}
The dimension of $Pic^0(\tilde{Z})$ is less than or equal to 
the dimension of $Prym(Z_B/X)$ which is $5$. 
\end{lemma}
\begin{proof}
The curve $X\cap B$ is a plane quintic so it has genus $g=6$, and its Pryms
have dimension $g-1$.
\end{proof}

\begin{corollary}
\label{A13}
The dimension of the locus of potentially obstructed bundles of type (ii) is $\leq 13$.
\end{corollary}
\eop

\begin{question}
\label{iiquest}
What is $A$?
\end{question}

If $A=0$ this would give a better bound on the dimension of 
the locus of potentially obstructed 
bundles of type (ii). On the other hand, if $A>0$ this would
indicate the presence of some interesting examples of irregular surfaces. So, in
any case it would be interesting to determine $A$. However, 
the answer would require first understanding the possible singularities
of $X\cap H$ for all possible quadrics $H\subset \pp ^3$, when $X$ is the fixed
general quintic surface. This goes beyond the scope of the present paper.

\subsection{Case (i)---generalities}

Suppose $P\subset X$ is a zero-dimensional subscheme. Recall that
$P$  {\em satisfies the Cayley-Bacharach property for quadrics in $X$} if it is nonempty, and for any subscheme $P'\subset P$ with $\ell (P/P')=1$,
the map 
$$
H^0(X,\Jj _{P'}(2))\rightarrow H^0(X,\Jj _{P}(2))
$$
is an isomorphism. In other words $P$ and $P'$ impose the same conditions
on quadrics on $X$. Given that 
$$
H^0(X,\Oo _X(2))\cong H^0(\pp ^3, \Oo _{\pp ^3}(2)),
$$
it is equivalent to say that $P$ satisfies the Cayley-Bacharach property 
for quadrics in $\pp ^3$. 

\begin{corollary}
\label{extensionexists}
Suppose $P$ is a zero-dimensional subscheme. Then there exists an exact sequence
\eqref{iseq}
$$
0\rightarrow \Oo _X(-1)\rightarrow E \rightarrow \Jj _P\rightarrow 0
$$
with $E$ locally free, if and only if $P$ is a l.c.i.\ Cayley-Bacharach for quadrics in $X$.
In particular, in case (i) of Proposition \ref{mainclassif}, we can add the condition that
$P$ is Cayley-Bacharach for quadrics. 

For any locally free extension fitting into \eqref{iseq} with $P$ nonempty,
the bundle $E$ is stable
and potentially obstructed with $c_1(E)=\Oo _X(-1)$ and $c_2(E) = d$. 
\end{corollary}
\begin{proof}
The first part is just Lemma \ref{totallysuper}. For the last paragraph,
the co-obstruction $\phi$ is given by the obvious factorization. 
The fact that $E$ is always stable comes from $H^0(E)=H^0(\Jj _P)=0$.
\end{proof}

Apply the discussion preceding Lemma \ref{gdef}, with
$L=\Oo _X(-1)$ and $L'=\Oo _X$:
\begin{equation}
\label{dualone}
E^{\ast} \cong E(1),
\end{equation}
giving
a symmetry of cohomology dimensions 
$$
h^0(E(1))= h^0(E^{\ast})= h^2(E(1)) = h^2(E^{\ast});
$$
$End ^0(E) = Sym ^2(E) (1)$; and there is an exact sequence
\begin{equation}
\label{esubseq}
0\rightarrow E \rightarrow End ^0(E) \rightarrow \Jj ^2_P(1) \rightarrow 0.
\end{equation}
By Lemma \ref{gdef}, there is a sheaf $\Gg$ fitting into the exact
sequences
\begin{equation}
\label{gquotseq}
0\rightarrow \Oo (-1)\rightarrow End ^0(E) \rightarrow \Gg \rightarrow 0.
\end{equation}
and
\begin{equation}
\label{gshortseq}
0\rightarrow \Gg \rightarrow E(1) \rightarrow \uExt ^2(\Oo _{2P}(1),\Oo _X)\rightarrow 0. 
\end{equation}

\begin{corollary}
\label{littlebigdiag}
Then
the map $Ad(\phi )$ fits into the diagram 
$$
\begin{array}{ccc}
& H^0(\Oo_{2P}(2))^{\ast} & \\
& \downarrow & \\
\cdots \;\; 
H^1(End ^0(E)) & \longrightarrow \;\; H^1(\Gg )\;\;  \longrightarrow & H^0(\Oo (2))^{\ast}   \;\; \cdots \\
& \downarrow & \\
&  H^1(E(1)) \;\; \rightarrow  &H^1(End ^0(E)\otimes K_X) \\
& \downarrow & \\
& 0 & 
\end{array}
$$
where the horizontal and vertical sequences are exact.
\end{corollary}
\begin{proof}
Apply Proposition \ref{bigdiag}. 
\end{proof}

\begin{lemma}
\label{notplanesit}
Suppose $P$ satisfies the Cayley-Bacharach property 
for quadrics, and is not contained in a plane.
Suppose given a vector bundle $E$ in an extension
of the form \eqref{iseq}. Then $h^0(E(1))=1$ so the exact sequence \eqref{iseq}
and in particular $P$, are determined by $E$. The extension class is determined up to
a scalar multiple. 

Let $d$ denote the length of $P$, then $h^1(E(1))=d-8$ and in particular $d\geq 8$.
\end{lemma}
\begin{proof}
The exact sequence \eqref{iseq} together with the exact sequence
for $\Jj _P\subset \Oo _X$ give 
$$
0\rightarrow H^0(\Oo _X)\rightarrow H^0(E(1))\rightarrow H^0(\Jj _P (1)),
$$
but the assumption that $P$ is not contained in a plane implies
$H^0(\Jj _P (1))=0$. Thus $\cc = H^0(\Oo _X)\cong H^0(E(1))$ 
and the map $\Oo _X(-1)\rightarrow E$ is unique up
to scalar multiple. Given $E$, fix the unique map and 
let $\Jj $ denote the quotient $E/\Oo _X(-1)$. Then
$\Jj ^{\ast \ast}$ is a line bundle, but since $det(E)\cong \Oo _X(-1)$ by hypothesis
we get $\Jj ^{\ast \ast} =\Oo _X$ so $\Jj $ is the ideal of a zero-dimensional subscheme 
$P$. The extension class is determined up to the scalar automorphisms of $\Oo _X(-1)$ and 
$\Jj$. 

To show that $h^1(E(1))=d-8$ consider long exact sequence
$$
0 \rightarrow 
H^1(E(1))\rightarrow H^1(\Jj _P(1))  
\rightarrow H^2(\Oo _X) \rightarrow H^2(E(1))\rightarrow 
H^2(\Jj _P(1)) \rightarrow 0.
$$
Note that 
$h^2(\Jj _P(1))=  h^2(\Oo _X(1)) =1$
and similarly by duality 
$h^2(E(1))= h^0(E^{\ast})= h^0(E(1))=1$. 
It follows that the last map of the above long exact sequence is an 
isomorphism, in particular its kernel is zero so the sequence becomes
$$
0\rightarrow H^1(E(1))\rightarrow H^1(\Jj _P(1)) \rightarrow H^2(\Oo _X) \rightarrow 0.
$$
On the other hand, by assumption $H^0(\Jj _P(1))=0$ and recall that
$H^1(\Oo _X(1))=0$, so we have the sequence
$$
0\rightarrow H^0(\Oo _X(1))\rightarrow H^0(P, \Oo _P(1))\rightarrow 
H^1(\Jj _P(1)) \rightarrow 0.
$$
Using $h^2(\Oo _X)= h^0(\Oo _X(1)) = 4$ and $P$ has $d$ points, this gives
$$
h^1(\Jj _P(1)) = d-4, \;\;\; h^1(E(1)) = d-8.
$$
This implies in particular that $d\geq 8$.
\end{proof}

\begin{lemma}
\label{2Plem}
Still assuming the Cayley-Bacharach property, 
let $2P$ denote the fat subscheme defined by $\Jj ^2_P$. Suppose that there
are no quadrics of $\pp ^3$ 
passing through $2P$. Then the obstruction is unique up to scalar, i.e. 
$h^2(End^0(E))=1$. Furthermore, if $\phi$ denotes a nonzero co-obstruction,
the map $Ad(\phi )$ factors as the composition 
of a surjection followed by an injection
$$
H^1(End ^0(E))\twoheadrightarrow H^1(E(1))
\hookrightarrow H^1(End ^0(E)\otimes K_X).
$$
\end{lemma}
\begin{proof} 
Calculate with the exact sequence from \eqref{esubseq}
$$
0 \rightarrow H^0(E(1)) \rightarrow H^0(End ^0(E)(1)) \rightarrow H^0(\Jj _P^2(2)).
$$
The assumption that there are no quadrics passing through $2P$ says that
$H^0(\Jj _P^2(2))=0$; among other things this implies that 
$P$ is not contained in a plane so the hypotheses of Lemma \ref{notplanesit} hold. 
We have seen above that $H^0(E(1))=\cc$ so 
$$
H^2(End^0(E))\cong H^0(End^0(E)(1)) ^{\ast}\cong \cc ,
$$
i.e. the obstruction is unique. 

Consider the diagonal map in the diagram of Corollary \ref{littlebigdiag},
$$
H^0(\Oo_{2P}(2))^{\ast} \rightarrow H^0(\Oo (2))^{\ast}.
$$
By Remark \ref{whatmap} this is dual to the restriction of quadrics to $2P$,
so the hypothesis implies that the dual map is injective. Hence the
diagonal map is surjective. A diagram chase then shows that the other
diagonal map 
$H^1(End ^0(E))\rightarrow H^1(E(1))$ is surjective. 

For the last map, use the long exact sequence for \eqref{esubseq} which
may be written
$$
H^0(\Jj ^2_P(2))\rightarrow H^1(E(1))\rightarrow H^1(End ^0(E)(1)) \rightarrow H^1(\Jj ^2_P(2)) .
$$
Again the hypothesis says that $H^0(\Jj ^2_P(2))=0$ so we get the required injectivity.
\end{proof}

\begin{corollary}
\label{obs1dichot}
In the situation of Lemma \ref{2Plem}, 
either 
$$
\dim _E(M )= 4d-20 
$$  
and the moduli space $M$ has a hypersurface singularity at $E$,
or else 
$$
\dim _E(M )= 4d-19 
$$  
and the moduli space is smooth at $E$. The quadratic term of the Kuranishi map 
has rank $h^1(E(1))=d-8$; in particular if $d>8$ then the moduli space has the expected dimension.
\end{corollary}
\begin{proof}
If the Kuranishi map vanishes identically then $M$ is smooth of dimension one more than
the expected dimension at $E$; if not then $M$ has the expected dimension at $E$
with a hypersurface singularity. 

The quadratic term of the Kuranishi map has the same rank as the rank of the linear
map $Ad(\phi )$. In the situation of Lemma \ref{2Plem}, this rank is $h^1(E(1))$.
If $h^1(E(1)) > 0$ then the Kuranishi map has to be nontrivial so $\dim _E(M)$
is the expected dimension.
\end{proof}

\begin{lemma}
\label{dimcount}
Suppose given a family of extensions of the form \eqref{iseq},
where the family of Cayley-Bacharach subschemes $P$ is faithfully parametrized by 
a variety $F$ of dimension $f$, and the general member $P$ of this
family imposes $c$ conditions on quadrics. Then $f\leq 2d$ and the total dimension
of the resulting family of extensions is $\leq f+d-c-1$. If the  $P$ are not in a plane 
then the dimension is equal to $f+d-c-1$. 
\end{lemma}
\begin{proof}
Note that $f\leq 2d$ since the full Hilbert scheme has dimension $2d$. 
For a given $P$ in the family $F$,
the dimension of the space of extensions is $d-c$ by Lemma \ref{totallysuper}, 
and taking into 
account the scalar multiple the full dimension of the family of bundles $E$
is  bounded by $f+d-c-1$.

If we are in the situations of Lemma \ref{notplanesit}, then for each 
potentially obstructed bundle near $E$ the subscheme $P$ is uniquely determined;
and the extension class is determined up to scalar multiple, so the dimension estimate
is an equality.  
\end{proof}

Let $\Sigma _d = \Sigma _d(X,\Oo _X(-1))$ denote the moduli scheme of extensions
of the form \eqref{iseq}. The previous lemma gives the dimension of strata in $\Sigma _d$. If $P$ is contained in a plane then $h^0(E(1))\geq 2$ and several extensions can
be associated to the same bundle $E$, in which case 
the map $\Sigma _d\rightarrow M_X(2,-1,d)$ 
may have positive dimensional fiber.

\begin{lemma}
\label{planesit}
In the situation of Corollary  \ref{extensionexists} including the hypothesis
that $P$ is Cayley-Bacharach for quadrics, suppose
$P$ is contained in, and spans a plane. Then the space
of obstructions has dimension $3$, and a general co-obstruction
corresponds to a reducible spectral curve whose determinant $\beta$ is the
square of the linear form defining the plane containing $P$.
\end{lemma}
\begin{proof}
Suppose $I\subset \pp ^3$ is a plane with $P\subset X\cap I$. 
The long exact sequence of \eqref{iseq} gives
$$
0\rightarrow \cc \rightarrow H^0(E(1))\rightarrow H^0(\Jj _P(1)) \rightarrow 0
$$
and $H^0(\Jj _P(1))=\cc$. Thus $H^0(E(1))=\cc ^2$ and we get an injection 
$$
\Oo _X(-1)^{\oplus 2}\hookrightarrow E.
$$
Using $E(1)\cong E^{\ast}$ we get dually $h:E\hookrightarrow \Oo _X^{\oplus 2}$.
From this we can easily construct a whole $3$-dimensional ${\mathfrak sl}(2)$ of 
maps $E\rightarrow E(1)$. The cokernel of $h$ is a sheaf of pure dimension $1$
whose cycle class has to be the same as $C$, on the other hand it goes
through the points $P$ and by hypothesis $P$ spans $I$. These imply
that the cokernel of $h$ is a rank one torsion-free sheaf on $C$,
so $E$ is obtained as an elementary transformation as in \S \ref{sec-reducible}.
The elements of the constructed ${\mathfrak sl}(2)$
of co-obstructions, are matrices vanishing along $C$ so they all have $\det (\psi )=\alpha ^2$ where $\alpha$ is the equation of $C$. 

Some further exact sequences show that the space of co-obstructions has dimension $\leq 3$:
the exact sequence corresponding to \eqref{esubseq} tensored with $K_X=\Oo _X(1)$ is
$$
0\rightarrow H^0(E(1))\rightarrow H^0(End ^0(E)(1))\rightarrow H^0(\Jj ^2_P(2))\rightarrow
H^1(E(1));
$$
the hypothesis that $P$ spans the plane $I$ implies that $2P$ is contained in a unique
quadric, namely $2I$, so
$h^0(\Jj ^2_P(2))=1$; and 
the long exact sequence of \eqref{iseq} gives
$$
0\rightarrow \cc \rightarrow H^0(E(1))\rightarrow H^0(\Jj _P(1)) \rightarrow 0
$$
with $H^0(\Jj _P(1))=\cc$. Putting these together gives a bound of $3$ for the dimension
of the space of co-obstructions. Since we have already constructed a three dimensional
space, the dimension must be $3$ and a general element is one of our constructed 
endomorphisms. This completes the proof. 
\end{proof}

\section{Examples of families for $d\leq 9$}
\label{sec6}

In this section we discuss the construction of
examples of potentially obstructed bundles on our very general quintic surface $X$, 
falling into case (i) of Proposition \ref{mainclassif}, in particular in the case
$d\leq 9$ when all bundles are of this type.  

To start off, if $D\subset \pp ^3$ is a line then $X\cap D$ has length $5$.
Looking at 
$4$ points i.e. a subscheme $P$ of length $4$ in $X\cap D$, this satisfies the Cayley-Bacharach property: 
vanishing along $D$ imposes only $3$ conditions on quadrics, and any three of the
points will do the job. This gives a construction of a stable bundle with
$c_2=4$, which is the lowest possible and 
in some sense epitomizes the further constructions below. 

\subsection{Euler characteristic considerations}
\label{sec-euler}

The Euler characteristic provides some useful information.
Suppose $E$ is a rank two bundle with $c_1=-1$ and $c_2=d$. For any $n$, noting that $\chi (\Oo _X(n)) = (15n^2 -15n +30)/6$,
and that the Chern character of $E$ is homologically 
equivalent to ${\rm ch} (\Oo _X) +{\rm ch}( \Oo _X(-1))$ minus $d$ points, Riemann-Roch gives
\begin{equation}
\label{euler}
h^0(E(n))-h^1(E(n)) + h^2(E(n)) = 5n^2 - 10 n + 15 -d.  
\end{equation}
For $n=1$ recall that $h^0(E(1)) = h^2(E(1))$, so
\begin{equation}
\label{euler2}
h^0(E(1)) = \frac{h^1(E(1)) + 10 -d}{2}.
\end{equation}
If $d\leq 9$ this is always positive, so every point has to be
a type (i) potentially obstructed point of the moduli space. Note along the way that if $d$ is any odd
number then $h^1(E(1)$ must be nonzero. 

Using the exact sequence \ref{iseq} gives
$$
h^0(E(1)) = 1 + h^0(\Jj _P (1)).
$$
In particular, if $d\leq 3$ then $h^0(E(1))\geq 4$ and 
$h^0(\Jj _P (1))\geq 3$, and a $3$-dimensional subspace of the linear
sections defines a point. 
So, $P$ would consist of at most a single point,
but $d=1$ is ruled out by the Bogomolov-Gieseker inequality. Therefore 
$M(2,-1,d)$ is empty for $d\leq 3$. 

For $d=4,5$ we get $h^0(E(1))\geq 3$ so 
$h^0(\Jj _P (1))= 2$ (the cases $\geq 3$ being ruled out as above).
A $2$-dimensional subspace of linear sections
vanishes along a line $D\subset \pp ^3$. Thus $P\subset X\cap D$, and indeed
any subscheme of length $4$ or $5$ in the finite scheme $X\cap D$ 
gives a Cayley-Bacharach subset.

For $d=6,7$ we get $h^0(E(1))\geq 2$, and $h^0(\Jj _P (1))\geq 1$. 
The case $\geq 3$ is ruled out as before. If $h^0(\Jj _P (1))=2$ then
$P$ would be contained in a line, but our general quintic surface $X$ doesn't contain
any lines so $X\cap D$ would have length $5$, not enough to contain $P$. Thus
we conclude that $h^0(E(1))= 2$ and $h^0(\Jj _P (1))=1$: the subscheme $P$ is contained
in a unique plane $H\subset \pp ^3$. 

The further discussion of these cases, as well as the cases $d=8,9$,
will be continued in the next section. 

For analyzing irreducibility in the cases $d\geq  10$ it will probably be useful to use
the Euler-characteristic considerations at $n=2$. In this case $h^2(E(2))= h^0(E)=0$ by
stability of $E$, so we get
$$
h^0(E(2)) = h^1(E(2))+15-d,
$$
in particular if $d<15$ then $E$ can be expressed as an extension
of the form 
$$
0\rightarrow \Oo _X(-2)\rightarrow E \rightarrow \Jj _Q(1)\rightarrow 0
$$
where $Q$ is a subscheme of length $d+10$ satisfying
the Cayley-Bacharach property for $\Oo _X(4)$. It would go beyond the scope of the present
paper to pursue this further, but
it might be useful for extending Nijsse's irreducibility results into the 
range $10\leq d \leq 15$.

\subsection{Points on the rational normal cubic curve}
\label{sec-rationalnormal}
 
One way to construct subschemes $P$ which satisfy Cayley-Bacharach for quadrics,
is by  taking subcollections of points in $X\cap N$ where $N\subset \pp ^3$ is
a curve. We start with a fairly easy version. 

If $N$ is a rational curve and $d> h^0(\Oo _N(2))$ then any
collection of $d$ points on $N$ satisfies Cayley-Bacharach for quadrics, indeed
any $d-1$ of the points already force a quadric to vanish on $N$, because
any collection of points satisfies the maximal rank property
for line bundles on $N\cong \pp ^1$. 

Some possible choices for $N$ would be a line as above, or two skew lines, or
three lines meeting in a rational stick-figure. These basic examples
suggest looking more generally at the rational normal cubic curve
$$
N= \{ [1:t:t^2:t^3]\} \subset \pp ^3
$$
which is embedded by the full linear system $|\Oo _{\pp ^1}(3p)|$. 
The restriction of a quadric to $N$ is a section of $\Oo _N(6p)$, so
points on $N$ can impose a maximum of
$$
h^0(N,\Oo _N(6p))= 7
$$
conditions on quadrics. As soon as there are $d\geq 8$ points, the resulting
subscheme satisfies Cayley-Bacharach. 

The choice of $N\subset \pp ^3$ has $f=12$ parameters, as can be seen in the 
following way. The group of automorphisms of $\pp ^3$
fixing a given $N$ is $PSL(2)\subset PSL(4)$ with the embedding given by identifying
$\cc ^4$ with
the symmetric cube of the standard representation of $SL(2)$: 
$$
\cc ^4 = Sym ^3(\cc ^2), \;\;\; SL(\cc ^2)\rightarrow SL(\cc ^4).
$$
Any other rational normal cubic is a translate of $N$, since they have the invariant
characterization as being given by complete linear systems of degree $3$ on $\pp ^1$.
The space of translates of $N$ is $PSL(4)/PSL(2)$ which has dimension $15-3=12$.

The number of conditions imposed by a subscheme $P$ is equal to $7$, so
by Lemma \ref{dimcount}, the dimension of the space of potentially obstructed bundles is
$d+4$. 

In the first case $d=8$, we get a $12$-dimensional family, indeed there is a unique
extension for each $P$. In this case by \eqref{ed} the expected dimension is
also $12$.

For the bundles in this family,
the dimension of the space of obstructions is $1$, indeed
whenever $P$ forces a quadric form to vanish on a curve in $\pp ^3$ then
quadrics vanishing on the double $2P$ of $P$ in $X$, must in fact vanish on the 
double of $P$ in $\pp ^3$ because they also vanish in transverse directions to $X$.
It now suffices to have $4$ points spanning $\pp ^3$ in order to insure that
there are no quadrics passing through $2P$. Then apply Lemma \ref{2Plem}.

Hence, the dimension of $M (2,-1,8)$ at a general $E$ in our family, is either 
$12$ or $13$; but in any case the Zariski tangent space has dimension $13$. 

\subsection{Points on normal elliptic curves}
\label{sec-normalelliptic}

The example of the preceding subsection can be generalized as follows. 

Suppose $Y$ is a smooth elliptic curve, with a divisor of degree $4$ denoted $\Oo _Y(1)$.
Notice that no basepoint is chosen on $Y$, and all divisors of a fixed nonzero degree
are related by translation automorphisms of $Y$, so the choice of $(Y,\Oo _Y(1))$ 
corresponds to a single parameter (the $j$-invariant of the elliptic curve).
Now $H^0(\Oo _Y(1))\cong \cc ^4$, and the choice of such an identification up to scalars
is a $15$-dimensional space. Any such choice gives an
embedding of $Y$ as a normal elliptic curve of degree $4$ in $\pp ^3$.
There is no connected subgroup of $PSL(4)$ fixing $Y$, so the dimension of the family of
all smooth normal elliptic curves of degree $4$, is $16$. Now, choose a collection $P'$ of 
$7$ points among the $20$ in $X\cap Y$.
This is a discrete choice out of $C^{20}_7$ possibilities. 
There is a unique nonzero section of $\Oo _Y(2)$ vanishing on these points, 
and this section has an additional zero denoted $q=q(Y,P')\in Y$. Let $F$ be the
family of all choices $(Y,P')$ such that $q(Y,P')\in X$ also. This is a codimension $1$
condition (and nontrivial, since $X$ is general), so $F$ has dimension $15$. 
We think of $F$ as the family
of $(Y,P)$ such that $P\subset X\cap Y$ has $8$ points and $\Oo _Y(P)\cong \Oo _Y(2)$. 

The projection from $F$ to the Hilbert scheme of subschemes $P\subset X$,
has as fiber the collection of all degree four elliptic curves $Y'$ passing through $P$.
Any such $P$ imposes $7$ conditions on quadrics of $\pp ^3$, and 
any of the $8$ points is superfluous. There is a $3$-dimensional space of quadrics
passing through $P$. 

We claim that for any elliptic curve $Y'$ of degree $4$ passing through $P'$,
two of the quadrics will vanish on $Y'$. Indeed, a ninth point on $Y'$  imposes
an additional condition, so there is a two dimensional space of quadrics passing
also through the ninth point, but a degree consideration shows that these must
contain $Y'$. We conclude that any point $(Y',P)$ in $F$ corresponds to a choice of
two dimensional subspace of the $\cc ^3$ of quadrics passing through $P$.
In particular, the fiber of $(Y',P)$ lying over $P$ has dimension $2$, so the family
of subschemes $P$ has dimension $13$.

General subschemes $P$ in this family satisfy the condition of Lemma \ref{2Plem},
so the space of obstructions is one-dimensional and the moduli space has dimension
$\leq 13$; thus our construction gives an irreducible component of $M (2,-1,8)$
of dimension $13$. At a general point (or in fact, any point at which the
obstruction space is unidimensional) the Kuranishi map vanishes entirely, that is
to say all higher order obstructions vanish and the moduli space is smooth but not of
the expected dimension.

It turns out that the $12$-dimensional locus of potentially obstructed
bundles obtained using the rational normal cubics, lies in the closure of the $13$-dimensional
family  constructed using elliptic curves. This can be seen by using the 
lower bound of Corollary \ref{dimsigmageq}.

\begin{corollary}
\label{geq13}
All components of the moduli space of bundles with $c_2=8$, have dimension $13$.
In particular, the $12$ dimensional component constructed in the previous subsection
is in the closure of the $13$ dimensional family constructed above. 
\end{corollary}
\begin{proof}
By Corollary \ref{dimsigmageq}, all components of $\Sigma _8$ have dimension $\geq 13$.
By \eqref{euler2}, all stable bundles with $c_2=8$ are in $\Sigma _8$, so all components of
the moduli space have dimension $\geq 13$. On the other hand, at any point 
where $P$ is not contained in a plane and not having a co-obstruction with
irreducible spectral covering, the dimension is $\leq 13$ because the expected
dimension is $12$ and the obstruction space has dimension $1$. 
The dimension of the space of bundles having co-obstruction with irreducible spectral
covering, is $\leq 13$ (Corollary \ref{A13}). 
The dimension of the space of extensions with $P$ 
in a plane, is bounded by $3$ for the choice of plane $U$, plus $8$ for the
choice of subscheme of $U\cap X$, plus $2$ for the choice of extension class since $c\geq 5$.
But in this case, the same bundle corresponds to a one-dimensional family of extensions, so
the space of such bundles has dimension $\leq 12$. This shows that all components have dimension $13$. 

The $12$-dimensional family constructed using a rational normal curve, therefore has
to generalize to a $13$-dimensional family, but $8$ general points on a rational normal curve can only generalize to $8$ points in general position for quadrics. 
This will be  shown in greater detail in Theorem \ref{d8irred} below, where we conclude that
the moduli space $M (2,-1,8)$ is irreducible and generically smooth. 
In other words, all stable bundles 
with $c_2=8$ are in the closure of the $13$-dimensional
family constructed above.
\end{proof}

\subsection{Nine points on an elliptic curve}

As we saw above, the family of elliptic curves of degree $4$ in $\pp ^3$ is $16$
dimensional. Denote by $F$ the family of $(Y,P)$ where $Y$ is such an elliptic curve
and $P\subset X\cap Y$ is a choice of $9$ out of the $20$ intersection points.
The choice of $P$ is discrete so $F$ has $16$ dimensions. The space of quadrics on
$Y$ has dimension $8$, so $P$ imposes $8$ conditions, and if $Y$ is in
general position (indeed, if it is not in the hypersurface considered previously)
then any $8$ points will also impose $8$ conditions so $P$ is Cayley-Bacharach. 
A general choice will satisfy this general position
property because $X$ was assumed general to start with.
The number of extra  points is $1$ so the extension class
is unique up to scalars. We obtain in this way a $16$ dimensional family
of potentially obstructed bundles; at a general point the dimension of the space of obstructions
is $1$. 
As before, the component constructed here is the only one. 

\begin{theorem}
\label{only9}
The moduli space $M (2,-1,9)$ is irreducible, $16$-dimensional. It is non-reduced
at the generic point, indeed at a general point
the moduli space is given by a single equation $x^2=0$.
\end{theorem}
\begin{proof}
The expected dimension is $16$. 
Consider first the $16$-dimensional family constructed above.
By Corollary \ref{A13}, the dimension of the space of potentially obstructed bundles 
of type (ii) is $\leq 13$. 
We can again argue that the dimension of $M (2,-1,9)$
at a general member $(E_0 , Y_0, P_0)$ of the family, is $16$. 
All nearby bundles $E_t$ would be potentially obstructed,
with one-dimensional obstruction space, so coming from subschemes $P_t$;
however, since $P_t$ must also be Cayley-Bacharach, it can still impose only
$8$ conditions on quadrics. We get a $2$-dimensional family $V_t$ of quadrics passing
through $P_t$, but this must be a smooth deformation of the $2$-dimensional family
of quadrics $V_0$ going through $P_0$, i.e. the equations of $Y_0$. Note that $Y_0$
is a complete intersection of any two elements forming a basis of $V_0$,
so the intersection of two basis elements of $V_t$ is again an elliptic curve of degree $4$,
$Y_t$ and we have $P_t\subset Y_t$. Thus, any deformation must remain inside our
given family, so we have constructed a $16$-dimensional component of
$M (2,-1,9)$. The points of this component are generically potentially obstructed, so 
they are in fact obstructed in the classical sense, i.e. we get a
non-reduced component. 

The quadratic term of the obstruction
map is nonzero: by Corollary \ref{obs1dichot},
$Ad(\phi )$
has rank $d-8=1$. Therefore the quadratic term
of the Kuranishi map is nonzero at a general point. Hence, at a general point
of this component the moduli space is given by $x^2=0$. 

To complete the proof, we need to say that all stable bundles 
with $c_2=9$ are in the closure of the $16$-dimensional
family constructed above.
Heuristically speaking, 
the special cases of $9$ points on a rational normal cubic, or $9$ points
arrayed on $2$ lines (in groups of $4$ and $5$) or $3$ meeting lines (in groups of $2$ or
$3$) are limits  of the general family we constructed here. Indeed, the elliptic curves
can degenerate, and since our $9$ points are general on $Y$, the points can go into
the various different components of the degeneration in various different ways.

The actual proof will be completed by Proposition \ref{d9cor} below, after a
detailed dimension count for the degenerate cases.
\end{proof}

\section{Dimension estimates and classification}
\label{sec7}

For $d\geq 10$ we show that $M_X(2,-1,d)$ is good, i.e.
all components are generically smooth of the expected dimension. 
This result is sharp: for $d\leq 9$ the moduli space is either generically 
non-reduced, or else has dimension bigger than the expected one, as can be seen
by the discussion of the previous chapter. 

We complete that discussion by showing that the moduli 
space is irreducible for $d\leq 9$, which also yields a pretty explicit
description of the irreducible components at their generic points as well as
some necessarily incomplete information about the smaller strata.

The question of irreducibility for $d\geq 10$ is left to the future. 
Nijsse \cite{Nijsse} proves irreducibility for $d\geq 16$, leaving open
the cases $10\ldots 15$. 
Another interesting question would be to understand
the relationship between our descriptions for $d\leq 9$, and the classification of ACM bundles of \cite{ChiantiniFaenzi}
in the cases $d=4,6,8$. 

\subsection{When $d\geq 10$}
$\mbox{ }$
\newline
Estimates for the dimension
of the potentially obstructed locus will show that the moduli space is good for $d\geq 10$. 
Recall that $\Sigma _d= \Sigma _d(X,\Oo _X(-1))$ 
denotes the moduli scheme of extensions of the form \eqref{iseq} with 
$\ell (P)=d$. 
A point can be written $(P,\xi )$ as in \eqref{xiseq}. 
By Corollary \ref{dimsigmageq} all irreducible components of $\Sigma _d$ have
dimension $\geq 3d-11$. 

Let $\Sigma _d^c$ denote the locally closed subset of $\Sigma _d$ consisting
of extensions such that $P$ imposes exactly $c$ conditions on quadrics. The maximal value $c=10$ corresponds to the case when $P$ is in general position with respect to quadrics.

\begin{lemma}
\label{twodminusone}
For $c\leq 9$ we have $\dim (\Sigma _d^c)\leq 2d-1$.
\end{lemma}
\begin{proof}
Consider the scheme of triples $(H,P,\xi )$ where $H$ is a quadric such that $P\subset H$.
The condition $(P,\xi )\in \Sigma _d^c$ means that the rank of $\epsilon $,
the evaluation map of quadrics on $P$, 
is $c$. Hence, there is a projective space of dimension $9-c$ of quadrics passing
through $P$, in other words the fibers of the map 
$$
\{ (H,P,\xi )\} \rightarrow \Sigma _d^c
$$
have dimension $9-c$. The dimension of the space on the left is therefore 
$\dim (\Sigma _d^c)+9-c$. On the other hand, we can estimate the dimension of this space
$\{ (H,P,\xi )\}$ by first choosing $H$ (in a projective space of dimension $9$), 
then noticing that $H\cap X$ is a curve
in a surface; by the result of Brian\c{c}on, Granger, Speder \cite{BGS}
the Hilbert scheme of subschemes of this curve, of length $d$, has dimension $\leq d$.
For each choice of $P$, the dimension of the 
space of choices of $\xi$ up to scalar, is $d-c-1$. Hence we get 
$$
\dim (\Sigma _d^c)+9-c \leq 9+d+(d-c-1),
$$
giving the stated estimate after subtracting $(9-c)$ from both sides.
\end{proof}

\begin{corollary}
\label{dimsigma1011}
We have $\dim (\Sigma _{10})= 19$; and for $d\geq 11$, the subset $\Sigma ^{10}_d$  is
dense in $\Sigma _d$, which is irreducible of dimension $3d-11$. 
\end{corollary}
\begin{proof}
The subset $\Sigma ^{10}_d$ consists of the $2d$ dimensional open set
of subschemes  $P$ in general position with
respect to quadrics, together with $\xi$ in a space of dimension $d-11$,
so $\dim (\Sigma ^{10}_d)=3d-11$. Notice that for $d=10$ this subset is empty
because such $P$ will not satisfy Cayley-Bacharach. For $d=10$, all 
the strata are of the form $\Sigma _{10}^c$ for $c\leq 9$ and these have dimension $\leq 19$
by Lemma \ref{twodminusone}. 
By Corollary \ref{dimsigmageq} each irreducible component has dimension $\geq 19$ so 
$\Sigma _{10}$ is pure of dimension $19$.

For $d\geq 11$, the strata $\Sigma _d^c$ have
dimension $\leq 2d-1$ but $2d-1<3d-11$, and by Lemma \ref{dimsigmageq} all irreducible
components of $\Sigma _d$ have dimension $\geq 3d-11$. Hence, the strata for $c\leq 9$ are
not irreducible components. Thus $\Sigma ^{10}_d$, which is smooth and irreducible
of dimension $3d-11$,
is dense in $\Sigma _d$. 
\end{proof}

\begin{theorem}
\label{goodness}
If $d\geq 10$ then all irreducible components of $M _X(2,-1,d)$ are good, i.e. generically smooth of the expected
dimension. 

If $d\geq 11$ then the $3d-11$-dimensional 
component $\overline{\Sigma}_d^{10}$, closure of the locus of potentially obstructed
bundles consisting of bundles of type (d) is nonempty, and is the biggest irreducible component of the singular locus of the  
moduli space. 
The singularities of $M$ along a general point of $\Sigma$ are ordinary quadratic
double points in the transverse direction. 
\end{theorem}
\begin{proof}
Use Corollary \ref{dimsigma1011}, and the estimate 
of Corollary \ref{A13} for the locus of potentially obstructed bundles of type (ii).
Since $13< 3d-11$ for $d\geq 10$, together these say that the dimension of the locus of potentially obstructed bundles is bounded by $3d-11$. 
But if $d\geq 10$ then $3d-11 <  4d-20$, which implies that the moduli space is good. 

If $d\geq 11$ then $\Sigma _d$ is irreducible by Corollary \ref{dimsigma1011},
and the components of the locus of potentially obstructed bundles of type (ii) have 
dimension $\leq 13 < 3d-11$.
This shows that $\Sigma_d=\overline{\Sigma} ^{10}_d$ is the unique 
irreducible component of largest dimension.

Recall now the description of the quadratic term in the obstruction map at a general
point of $\Sigma$ where there is a single co-obstruction $\phi$. The quadratic term is
a symmetric bilinear form on $H^1(End ^0(E))$ whose rank is equal to the
rank of the linear map $Ad(\phi )$ which factors as
$$
H^1(End ^0(E)) \rightarrow H^1(E(1))\rightarrow H^1(End ^0(E)\otimes K_X).
$$
Since all quadrics passing through $2P$ must vanish, by Corollary \ref{obs1dichot}, 
$Ad(\phi )$
and the quadratic term of the Kuranishi map
have rank $d-8$. 

The Zariski tangent space at $E$ has dimension equal to the expected dimension
plus the number of obstructions, i.e. $4d-19$. The component $\Sigma_d$
has dimension $3d-11$, so the transverse direction in the Zariski tangent space
has dimension $d-8$; it follows that the singularity is an ordinary double point
in the transverse direction, noting that the quadratic form has to vanish
in the directions along $\Sigma$. In other words, the equations for
$M$ are locally of the form
$$
x_1^2+\ldots + x_{d-8}^2=0
$$
in terms of a local coordinate system 
$x_1,\ldots , x_{4d-19}$ such that $x_1,\ldots , x_{d-8}$ are the coordinate functions defining $\Sigma$. 
\end{proof}

The result of Theorem \ref{goodness} improves upon the sharp bound asked for by O'Grady
in his Question, \cite[p. 112]{OGradyQuest}:
\newline
---``is $M(\xi )$ good if $\Delta _{\xi}> \rk (\xi )(p_g+1)$?''
\newline
For $\xi = (2,-1,d)$ we  have $\rk (\xi )=2$, $\Delta _{\xi}= d-\frac{5}{4}$
and $p_g=4$ for our
quintic surfaces
so the sharp bound asked for by O'Grady would say $d> 11 \frac{1}{4}$ i.e. $d\geq 12$. 
The result of Theorem \ref{goodness} improves this by $2$; and for $d=9$ the
moduli space is generically non-reduced so it isn't good. 
So, the bound is now completely sharp and the phase transition occurs at $c_2(E)= \rk (\xi )(p_g+1)=10$. The moduli space at $c_2=10$ looks like a particularly interesting case
for further study.

\subsection{Classes of Cayley-Bacharach subschemes}

Motivated by the examples described in \S \ref{sec6} for $d\leq 9$, we state some finer classification results for the
potentially obstructed bundles of type (i). The first classification is by the number of
conditions imposed on quadrics. 

\begin{proposition}
\label{iclassif}
Suppose $X\subset \pp ^3$ is a sufficiently general quintic surface. 
Suppose $E$ is a potentially obstructed 
vector bundle on $X$ of type (i) from Proposition \ref{mainclassif} fitting into an extension \eqref{iseq} with
a subscheme $P$ of length $d$, satisfying Cayley-Bacharach for quadrics. 
Let $c$ be the number of conditions imposed by $P$ on quadrics, and 
let $Y\subset \pp ^3$
be the subscheme defined by the quadrics vanishing on $P$. Then either:
\newline
(a)---$c\leq 7$ and $Y$ is zero-dimensional, in this case $d\leq 8$; 
\newline
(b)---$c\leq 7$ and $\dim (Y)\geq 1$;
\newline
(c)---$c=8$ and $Y$ is a plane union a line;
\newline
(d)---$c=8$, $d\leq 20$ and $P$ is supported on a zero-dimensional intersection
$X\cap Y$ where $Y$ is a 
possibly degenerate genus one degree four curve
which is the complete intersection of two set-theoretically transverse quadrics;  
\newline
(e)---$c=9$, $d\geq 10$ and $P$ is supported on $Z=X\cap H$ for a unique quadric
surface $H$; or
\newline
(f)---$c=10$ i.e. $P$ is in general position with respect to quadrics, and $d\geq 11$.
\end{proposition}
\begin{proof}
Let $V:= H^0(\Jj _P(2))\subset H^0(\Oo _{\pp ^3}(2))\cong \cc ^{10}$ be the
space of quadrics vanishing on $P$. 

If $c=10$ then $d\geq 11$ by the Cayley-Bacharach condition and we are in case (f), which was considered for example
by O'Grady \cite[(3.29)]{OGradyBasic}.

If $c=9$ then $d\geq 10$ and
$\dim (V)=1$,
and the unique element of $V$ up to scalars, determines a unique quadric $H$ such that
$P\subset X\cap H$. This gives (e).

Suppose $c= 8$, then $\dim (V)\geq 2$ and we can choose two linearly independent
elements. These correspond to quartics $H_1$ and $H_2$. 
Suppose first of all that $Y=H_1\cap H_2$ has dimension $1$. Then it has degree $4$
and is a possibly degenerate version of an elliptic curve; and $P\subset Y\cap X$.

The fact that $X$ is general implies that $Y$ is not contained in $X$ and indeed meets
$X$ in a finite subscheme. This intersection $X\cap Y$ has length $20$, so the length of $P$ is
$\leq 20$. This gives (d).

Suppose in the previous situation, on the other hand, that 
$H_1$ and $H_2$ contain a common component of dimension $2$ in $\pp ^3$,
which must be a plane $I$. We can write $H_i=I\cup U_i$ with $U_i$ also being
distinct planes; then $P\subset H_1\cap H_2= I\cup D$ where $D=U_1\cap U_2$ is a line.
This gives (c).

If $c\leq 7$ then tautologically either (a) or (b). In case (a)
$Y$ is contained in a zero dimensional intersection of three quadrics which has length $8$
so $d\leq 8$. 
\end{proof}

\begin{lemma}
\label{dimcount1}
In all but the first two cases, we have the following estimates for the dimension 
$e$ of the corresponding stratum in $\Sigma ^c_d$: 
Case (c), $e\leq 2d-4$; Case (d), $e\leq 2d-1$; Case (e), $e\leq 2d-1$, 
Case (f), $e= 3d-11$.
\end{lemma}
\begin{proof}
Let $f$ be the dimension of the space of subschemes $P$ (i.e. of some irreducible component
of the space of subschemes satisfying Cayley-Bacharach); and $c$ is the number of conditions
imposed by $P$ on quadrics. Recall from Lemma \ref{dimcount} that $e\leq f+d-c-1$.

In Case (c), if $U$ is the plane in $Y$ and $U'$ is a plane containing the line,
then $P \subset X\cap (U\cup U')$ and this is a plane curve in $X$. Furthermore $U'$ can be chosen from amongst a family of dimension $2$, so the number of choices for $U\cup U'$ is $\leq 5$. 
The BGS estimate \cite{BGS}
says that, given $U\cup U'$ the dimension of the space of choices of $P$ is $\leq d$,
so $f\leq 5+d$ and $e\leq 2d-4$. 

For Cases (d) and (e) use the estimate of Lemma \ref{twodminusone}.
For Case (f) use Corollary \ref{dimsigma1011}.
\end{proof}

The following list of sub-cases of (d) will be useful below, and will also
help for (b). 

\begin{lemma}
\label{dcases}
In case (d) the curve $Y$, complete intersection
of two quadrics $H_1$ and $H_2$, is of one of the following types:
\newline
(d1)---a curve of genus $1$ with at most ordinary double points but no rational tails,
normally embedded by a divisor of degree $4$ having positive degree on each component;
\newline
(d2)---two rational curves of degree $2$ meeting at a tacnode;
\newline
(d3)---a cuspidal curve on a smooth quadric surface;
\newline
(d4)---four lines emanating from a single point, whose directions are in general position; 
\newline
(d5)---a double line, whose double structure is contained in a plane, plus two 
other lines not in the same plane, emanating from a single point;
\newline
(d6)---a double line, whose double structure is contained in a plane, plus a smooth rational
curve of degree $2$ in a transverse plane, tangent to the first plane at the intersection point;
\newline
(d7)---a skew double line (i.e. one 
whose double structure turns), plus two skew lines meeting the 
double structure in the given tangent directions at two distinct points; or
\newline
(d8)---contained in a double plane
\end{lemma}
\begin{proof}
Let $V$ be the subspace of quadrics spanned by the equations of 
$H_1$ and $H_2$, defining a linear system of quadrics. Since $Y$ is the base locus
of this linear system, it follows that the singularities of a general quadric $H$ in the 
system, are contained in $Y$. We may assume that the linear system doesn't
contain a double plane, otherwise we get (d8). If the general member were the union of two planes, then they
would meet along a line $D\subset Y$ but the linear system would be all unions of two planes
passing through $D$; this would contain a double plane contradicting the hypothesis.
So, the general quadric in the linear system is either smooth or a cone. If all quadrics in
$V$ are cones, then choosing two, we see that the vertex of each one is on the other,
so they share a common line. One possibility is that the vertices are all the same.
Then $Y$ is a cone over the intersection of two quadrics in $\pp ^2$, and we get
to cases (d4) or (d5). 

If the vertices can be distinct, 
looking at all other possibilities, we see that the vertices would
always be on the common line $L$. If the tangent planes to the two quadrics along $L$ are distinct, then a general linear combination will be smooth; so we may assume that the
tangent planes are the same. Then, $Y$ contains a double line whose double structure is contained in this tangent plane. The remaining components of $Y$ form a degree $2$ curve.
The remaining curve can not consist of two skew lines, because all the lines in a cone
go through a single vertex. So the remaining degree $2$ curve has to be 
contained in a plane $I$. If the plane contained the line then $Y$ would
be contained in $2I$ contradicting the hypothesis, so $I$ is transverse to $L$
and $Y$ is the double line plus a quadric curve in $I$. The quadric curve is defined by
the intersection of a general element of our linear system with $I$, so it must contain
the double point defined by the intersection of our double line with $I$; this gives
cases (d5) or (d6) and completes the treatment of the case where the general member
of the linear system is a cone.

The leftover case is when the linear system contains a smooth quartic surface
$Q_1\cong\pp ^1\times \pp ^1$. In this case $Y$ is a $(2,2)$-curve on $Q_1$ and
we get either (d1), (d2), (d3), (d7)  or (d8).
\end{proof}

In order to count dimensions, we need to investigate more closely
Case (b). 
Suppose $P\subset \pp ^3$ is a subscheme of length $d$ satisfying the 
Cayley-Bacharach property for quadrics. Let $V=H^0(\Jj _P(2))$ and let
$Y\subset \pp ^3$ be the subscheme defined by $V$. Let $Y_1\subset Y$ be the
subscheme whose ideal is 
generated by all sections of $\Oo _Y$ with support of dimension $0$. In other words,
$Y_1$ is the part of $Y$ of pure dimension $1$.  
Assume here that $\dim (V)\geq 3$, i.e. the number $c$ of conditions which $P$ imposes
on quadrics is $c\leq 7$. Choose two of the quadrics $H_1$ and $H_2$.
It follows from Lasker's theorem
\cite[p. 314]{EisenbudGreenHarris} that $Y_1$ cannot
be all of $H_1\cap H_2$, so $deg(Y_1)\leq 3$.

\begin{lemma}
\label{PinY}
Case (b) divides into the following possibilities:
\newline
(b1)---$Y_1$ is a plane;
\newline
(b2)---$Y_1$ is a line;
\newline
(b3)---$Y_1$ is a planar double line;
\newline
(b4)---$Y_1=Y'_1\cup Y''_1$ is a union of two skew lines;
\newline
(b5)---$Y_1$ is a skew double line;
\newline
(b6)---$Y_1$ is a fat triple line, that is $\Jj _{Y_1/\pp ^3}= \Jj _{L/\pp ^3}^2$;
\newline
(b7)---$Y_1$ is a union of three distinct lines;
\newline
(b8)---$Y_1$ is a union of a planar double line plus a transverse line;
\newline
(b9)---$Y_1$ is a union of a skew double line plus a tangent line;
\newline
(b10)---$Y_1$ is a reduced conic in a plane; 
\newline
(b11)---$Y_1$ is a union of a smooth planar conic, with a line meeting it at a smooth point;
or
\newline
(b12)---$Y_1$ is a rational normal cubic.

In all of these cases, the Cayley-Bacharach (CB) condition on $P$ forces
$P\subset Y_1$.
\end{lemma}
\begin{proof}
The present list of cases can be extracted from the enumeration of cases in Lemma \ref{dcases} above,
as the possible subschemes of degree $\leq 3$.

We show by cases that $P\subset Y_1$. 

Case (b1): call the plane $U:=Y_1$. It imposes $6$ conditions on quadrics.
Suppose $P$ contains a point $z_1\not \in U$, or
an embedded point sticking out of $U$. By CB, it must contain another point 
$z_2$ not in $U$, but together these two impose $2$ more conditions  
giving $c\geq 8$, a contradiction. So, in case (b1) we conclude $P\subset Y_1$.

Case (b2): call the line $L:=Y_1$. It imposes $3$ conditions on quadrics.
Suppose $z$ is a point of $P$ not in $L$,
or the trace of a multiple point sticking out of $L$. These define a plane $U$.
If $P$ contains another different point of $U$, then $L$ together with these two points
would define a unique quadric in the plane; vanishing along $P$ would therefore 
imply vanishing along this additional quadric (which could be a doubled structure along $L$);
but that contradicts the assumption (b2). We conclude that no two other points of 
$P$ can be coplanar with $L$. Similarly, no three other points can be colinear otherwise
$Y_1$ would contain that line too.
Now, given a point $z_1$ of $P-L$, by CB there has to
be a second point $z_2$ of $P-L$; these define a line skew to $L$. 
But there are quadrics vanishing on one but not the other of these points, so there
has to be a third point $z_3$, not colinear with $z_1$ and $z_2$, and not coplanar with
either point and $L$. In particular, $z_1,z_2,z_3$ define a plane $U'$ not passing
through $L$. Let $z_0=U'\cap L$. The four points $z_0,\ldots , z_3$ don't satisfy 
CB in the plane, so there has to be another point $z_4$. Now $z_4\not\in U'$ since
otherwise $z_0, \ldots , z_4$ would define a quadric in $U'$ which would have to be a part of $Y_1$, contradicting
(b2). But if the $z_4$ was not in the plane, then it would define a condition
again independent of $z_1,\ldots , z_3$. By CB we would need another point $z_5$.
No four of the five points $z_1,\ldots , z_5$ are coplanar by the same
argument as for $z_4\not\in U'$. But $z_5$
is not in the union of the plane $U'$ defined by $z_1,z_2,z_3$ and the plane $U''$ defined
by $L$ and $z_4$, so $z_5$ imposes yet another condition on quadrics, giving at least
$8$ conditions altogether. This contradiction shows that $P\subset Y_1$.

Case (b3): let $L$ be the reduced line and $U$ the plane, so $Y_1=2L$ in $U$.
Vanishing along $Y_1$ imposes $5$ conditions on quadrics. If there is another point 
$z\in P$ and if $z\in U$ then vanishing on $z$ would imply vanishing on $U$,
contradicting (b3). If $z_1\not \in U$ then it defines an independent condition,
so by CB there must be another point $z_2\not \in U$. But $z_2$ also defines an
independent condition, so again by CB there must be a third point $z_3\not \in U$.
If $z_1,z_2,z_3$ are colinear then quadrics would vanish along a different line,
contradicting (b3). If they are not colinear, then we can choose
a plane $U'$ containing $z_1$ and $z_2$ but not $z_3$, proving that $z_3$ imposes
an independent condition. Thus $c\geq 8$, a contradiction which shows that $P\subset Y_1$.

Case (b4): write $L:= Y'_1$ and $M:= Y'_2$. Vanishing along $Y_1=L\cup M$ imposes 
$6$ conditions. If $z_1\in P$, $z_1\not \in Y_1$ then let $U$ be the plane passing through
$L$ and $z_1$. Let $z_0=M\cap U$. These two points together with a line, define a quadric
in $U$ which must contain any quadric of $\pp ^3$ passing through $P$, contradicting
(b4) and showing that $P\subset Y_1$. 

Case (b5): the skew double structure along the reduced line $L$ corresponds to
an expression $Y=2L$ within a smooth quadric hypersurface $H\subset \pp ^3$. The quadrics
passing through $2L$ on $H$ are curves of type $(2,0)$, that is to say they consist
of two lines in the family of lines on $H$ transverse to $L$. Vanishing along $2L$ imposes
$6$ conditions on quadrics. Suppose $z_1\in P$, $z_1\not \in 2L$. 
Let $U$ be the plane passing through
$L$ and 
$z_1$. The intersection $U\cap (2L)$ is a line plus an embedded point somewhere along the line
but not at $z_1$. Thus $Y_1$ would contain the plane conic passing through these points.
Note that if $z_1$ is itself embedded along $L$ but not in $2L$, the conic
would be a double structure on $L$ distinct from $2L$, in any case contradicting (b4).
We conclude that $P\subset Y_1$. 

Case (b6): the ideal of $Y_1$ is the square of the ideal of a line $L$. 
Vanishing along $Y_1$ imposes $7$ conditions on quadrics. The pencil of quadrics
passing through $Y_1$ consists of all products of two planes passing through $L$;
it has no other base points so any other point $z_1$ would impose a further independent
condition giving $c\geq 8$. We conclude that $P\subset Y_1$. 

Cases (b7), (b8) and (b9): The number of conditions is $9$ minus the number of intersection points.
We conclude that there must be at least two intersection points or a triple intersection. 
In case of a double line, they are said to intersect when the double structure is 
planar, and a line intersecting a double line has one intersection point if it is
transverse to the tangent plane, or two intersection points if it is in the tangent plane
at the intersection point. Note moreover 
that the three lines cannot be coplanar otherwise quadrics vanishing
on them would vanish on the plane; thus the number of intersection points is exactly two. 
In particular, in (b9) the double line is skew so the other line should be tangent
to it; in (b8) the double line is planar so the other line should be transverse to the plane.
Once again, these situations all give exactly $7$ conditions, and any point outside
of $Y_1$ imposes an independent condition, so $P\subset Y_1$. 

Case (b10): suppose $Y_1\subset U$ is a reduced conic, that is either a smooth conic
or a union of two lines, in a plane $U$. Vanishing on $Y_1$ imposes $5$ conditions. 
If $P$ contained another point of $U$ then
$Y_1$ would contain all of $U$, contradicting (b10). If $z_1\in P$ but $z_1\not \in U$
then arguing as previously, we would get a second point $z_2$ by CB, then a third
point $z_3$ with the three points non colinear.  But this gives a total of $c\geq 8$ 
conditions, a contradiction showing that $P\subset Y_1$.

Cases (b11) and (b12): vanishing on $Y_1$ imposes $7$ conditions and the quadrics passing through
$Y_1$ define it, so we conclude that $P\subset Y_1$. 
\end{proof}

Estimate now the dimensions of the strata corresponding to these cases. 
It wouldn't be too hard to get an estimate of the form $e\leq 2d-1$, but for 
$d=8$ we need to know that $e\leq 12$. 

\begin{lemma}
\label{inYdims}
The dimensions $e$ of the strata in the moduli space of extensions $\Sigma _d$
parametrizing $(P,\xi )$ with subscheme 
$P\subset X$ in the previously mentioned cases (b1)--(b12), are bounded by
$$
e\leq \max (d+4, 2d-4).
$$
This gives the same bound for the dimension of the image of the map 
$\Sigma _d\rightarrow M_X(2,-1,d)$.
\end{lemma}
\begin{proof}
The moduli scheme of extensions $\Sigma _d$ has a stratification where, on each stratum,
one of the cases discussed above holds. We are looking here only at the strata 
corresponding to case (b) of Proposition \ref{iclassif}, and they are assumed 
broken up according
to the cases (b1)--(b12). In the course of the arguments and without too much further mention,
we assume that the stratification is furthermore refined in various ways so that things like
the order of contact between $Y_1$ and $X$ may be fixed. 
Use Lemma \ref{dimcount} and the Brian\c{c}on-Granger-Speder estimates \cite{BGS} throughout.

Case (b1): $f\leq d+3$, $c=6$, $e\leq 2d-4$;

Case (b2): $f=4$, $d=4$ or $5$, $c=3$, $e\leq 5$;

Case (b4): $f=8$ because the choices for $P\subset Q=(L\cup M)\cap X$
are discrete; and $c=6$ so $e= d+1$.

Case (b6): the number of choices for $P$ is less than $d$, and the number
of choices of the line is $4$ so $f\leq d+4$; and $c=7$ so 
$e\leq 2d - 4$. 

Case (b7): we have $c=7$ and the number of choices for $Y_1$ is
$10$. The three lines are distinct and not coplanar. 
For each double point, there is at most one variable
of choice of $P\subset Q=Y_1\cap X$ since $Q$ is contained in a planar double point on 
$X$. However, this imposes that $X$ pass through the vertex
in question, taking out a condition. So, in any case $f\leq 10$ and $e\leq d+2$.

Case (b10): we have $c=5$ and the space of choices of $Y_1$ has dimension $8$.
At a potiential double point there is at most one additional choice for $P$, 
but that only applies if this is a point of tangency of the plane with
the quintic $X$, which reduces the dimension by $3$; so in any case
$f\leq 8$ and we get $e\leq d+2$. 

Case (b11): we have $c=7$, the dimension of the space of choices of $Y_1$ is
$11$, and there is at most one additional parameter for the choice of $P$
at the double point; but as before this imposes a condition
of tangency with $X$ so in any case $f\leq 11$ which gives $e\leq d+3$. 

Case (b12): the space of choices of $Y_1$ has dimension $12$, the choice of $P$ is 
then discrete, and $c=7$
so we get $e\leq d+4$.

Consider now the cases (b3), (b5), (b8) and (b9) involving double lines.
Denote by $L$ the line and $2L$ the double structure. For (b8), 
(b9) denote by $D$ the other line.

Set $Q:= Y_1\cap X$ and write the decomposition into local components $Q=\bigcup _zQ_z$.
There are three types of points $z$: the vertex $0$ situated
at $L\cap D$, points on $L- \{0\}$, and points on $D-\{ 0\}$. Some of these subsets
(including that of the vertex) might be empty. 
Let $q_z:= \ell (Q_z)$, $q'_z:= \ell (Q_z\cap L)$, and $q''_z:= \ell (Q_z\cap D)$.
These may be zero for example if $z$ doesn't lie on one or the other of $L$ or $D$,
and of course in cases (b3) and (b5) $q''_z=0$ by convention.
Similarly, for each choice of $P$ let $P_z=P\cap Q_z$ and 
put $p_z:= \ell (P_z)$, with similar notations $p'_z$ and $p''_z$. 

There are various relations such as $q_z/2 \leq q'_z\leq q_z$ for $z\in L-\{ 0\}$,
and of course $p_z\leq q_z$, $p'_z\leq q'_z$, $p''_0\leq q''_0$; also
$\sum _{z\in D}q''_z = 5$, $\sum _{z\in L}q'_z = 5$. 

We recall some of the the estimates of 
Brian\c{c}on, Granger and Speder for the local dimension
$\delta (Q_z,p_z,p'_z,p''_z)$ of subschemes of length $p_z$ in $Q_z$, with
multiplicities $p'_z$ and $p''_z$ of intersection with $L$ and $D$ respectively. 
If $\mu _z$ denotes
the biggest multiplicity of a planar fat point contained in $P_z$ then 
$$
\delta (Q_z,p_z,p'_z,p''_z)\leq p_z -\mu _z.
$$
In our case as long as $p_z>0$ we have $\mu _z=1,2$, as no bigger fat point fits inside $Y_1$.
If $p'_z\geq 2$ and $p'_z\geq 2$ then $\mu _z=2$, otherwise $\mu _z=1$.

On the other hand, if $Q_z$ is curvilinear i.e. contained in a smooth curve, then 
the space of choices of $P_z$ is discrete and $\delta (Q_z,p_z,p'_z,p''_z)=0$.
This will be the case for any point of $D - \{ 0\}$ and also for any point $z\in L-\{0\}$
where $X$ is not tangent to the double structure. Also in case (b8)
this will be the case at the vertex of $Y_1$ if $X$ is not tangent. 

Write $\delta (Q_z,p_z,p'_z,p''_z)= p_z - n_z$ where $n_z \geq \mu _z$,
with $n_z = p_z$ in the curvilinear case. 

In general we will partition the space of choices of $(Y_1,P_z)$ into subfamilies according
to various conditions of tangency or other multiplicities including 
tangencies with $X$ and the $p'_z$ and $p''_z$
above. It suffices to estimate the dimension $f$ of the family in each case separately.
Denote by $\varphi$ the dimension of the stratum in the space of choices of $Y_1$
with given conditions of intersection with $X$, and denote by $Q_z$ the local 
intersections for some member $Y_1$. Given the choice of $Y_1$, there are finitely
many points $z$ in the support of $Y_1\cap X$ so we don't need to consider the
variation of $z$, so 
$$
f\leq  \varphi + \sum _z \delta (Q_z,p_z,p'_z,p''_z) = \varphi + d - \sum _zn_z.
$$
For brevity we denote $\delta (Q_z,p_z,p'_z,p''_z)$ by $\delta _z$.

Consider the case (b3) of a planar double line, 
with $c=5$. If $X$ is nowhere tangent to $Y_1$ then
$\varphi =5$ and $\delta _z=0$ so we get $e\leq d-1\leq 2d-5$. Suppose
$X$ is tangent to $Y_1$ at one or more points. The family of planar double lines
tangent to $X$ has dimension $\varphi = 3$ and 
$\sum _zn_z \geq 1$ so $f\leq d+2$, giving $e\leq 2d-4$. 

Consider the case (b5) of a skew double line, 
with $c= 6$. Again, if $X$ is nowhere tangent 
to $X$ then the dimension of the space of choices of $Y_1$ is $4$ for the line $L$,
plus $3$ for the skew structure, that is $\varphi = 7$. 
As $\delta_z=0$ this gives $f=7$ and  $e\leq d$. 
Suppose $X$ is tangent at some point $z$ with $p'_z=2$. 
The number of choices of a tangent line $L$ is $3$,
the number of choices of a skew double structure with given tangent plane at that
tangent point is $2$, this gives $\varphi \leq 5$. Since there must be
a point distinct from $z$ in $|P|$ we have $\sum _zn_z \geq 2$,
so $f\leq d+3$ and $e\leq 2d-4$. Suppose $X$ is tangent at a point $z$ and
$p'_z\geq 3$. Then the tangent line $L$ to $X$ has a higher order of contact.
The dimension of space of choices of $L$ is $\leq 2$ so $\varphi \leq 4$
in this case. Still  $\sum _zn_z \geq 1$ so $f\leq d+3$ and $e\leq 2d-4$. 

Consider the case (b8) of a planar double line plus a transverse $D$, with $c=7$.
The dimension of the space of choices of $Y_1$ is $5$ for the planar double
structure, plus $1$ for the vertex $0$ along $L$ plus $2$ for the line $D$
coming out of the vertex, thus $\varphi = 8$ if there are no constraints on $Y_1$. 
The CB condition implies that $\ell (P\cap D)\geq 4$. 

Consider the stratum of choices
where $X$ is nowhere tangent to the double structure. Then $Q_z$ is contained
in the transverse intersection of the plane with $X$ at each $z\neq 0$, i.e. $Q_z$ are
all curvilinear. For these points $\delta _z=0$.
At the vertex, note that $Y_1$ is contained in a union of
two planes transverse to $X$, so $Q_0$ is contained in an ordinary double point.
In this case $\delta _0\leq 1$. Putting these together gives
$f\leq 9$, whence $e\leq d+1$.

Consider the stratum of choices where $X$ is not tangent to the double
structure at the vertex, but is tangent at some other point of $L$.
The family of choices of $Y_1$ has dimension $3$ for the tangent line $L$,
plus $1+2$ for the vertex and transverse line $D$. Thus $\varphi \leq 6$.
The points $z$ on $X\cap D$
contribute $\delta _z =0$ away from the vertex, and $\delta _z\leq 1$ at the
vertex (since, as in the previous paragraph, $X$ is not tangent to the double
structure at the vertex so $Q_0$ is contained in an ordinary double point).
There are at least $4$ points counted with multiplicity along $D$, contributing only
$\delta _0\leq 1$. The length of the piece of $P$ supported away from $D$ is $\leq d-4$, so
the sum of
$\delta _z$ for this part of $P$ is $\leq d-5$ (using $\sum n_z\geq 1$),
we get altogether $\sum _{z\in |P|}\delta _z\leq d-4$. Hence $f\leq d+2$ and $e\leq 2d-6$.

Consider the stratum of choices where $X$ is tangent to the double structure at
the vertex $0$. Here there are $3$ choices for $L$ tangent to $X$ and $2$ choices
for $D$ coming out of the tangent point, so $\varphi \leq 5$. On the other hand,
$P$ must contain a subscheme of length $\geq 4$ along $D$, but since
$X$ is transverse to $D$, it contains a subscheme of length $\geq 3$ along
$D-\{ 0\}$. This part of $P$ contributes $\delta =0$. The remainder of $P$,
that is the part supported set-theoretically along $L$, has length $\leq d-3$.
For this part we have $\sum n_z\geq 1$ so 
$\sum \delta _z\leq d-4$, hence $f\leq d+1$ and $e\leq 2d-7$.
This completes the case (b8). 

Consider the case (b9) of a skew double line plus a tangent $D$, with $c=7$. 
Look first at the strata where $X$ is not tangent to the double
structure at the vertex $0$. The dimension of the space of choices of $Y_1$ is
$4$ for the line, plus $3$ for the skew structure, plus $1$ for the vertex, plus 
$1$ for the tangent line, giving $\varphi = 9$. 
Note that $Q_0$ is curvilinear, 
being contained in the intersection of $X$ with a smooth quadric surface containing $Y_1$. 
Thus, it contributes $\delta _0=0$. Similarly for any other $z\in D$, $\delta _z=0$.
Since $\ell (P\cap D)\geq 4$, the part of $P$ which doesn't touch $D$ has length $\leq d-4$.
If this part is nonempty then its $\sum n_z\geq 1$ so $\sum \delta _z\leq d-5$.
This gives $f\leq d+4$ hence  $e\leq 2d-4$. 

Suppose from now on that $X$ is tangent to the double structure at the
vertex. The number of choices is $3$ for  the tangent line, plus $2$ for the
double structure with given tangent plane at the vertex, plus $1$ for the 
line $D$ emanating from $0$ in the tangent plane, thus $\varphi \leq 6$. 
If $p''_0\leq 1$ then the part of $P$ touching $L$ has length $\leq d-3$,
and we get $\sum \delta _z\leq d-4$ hence $f\leq d+2$ and $e\leq 2d-6$. 
Similarly, if $p'_0\leq 1$ then there is at least one other point $z$ in $P\cap L$
so $\sum n_z\geq 2$ and $\sum \delta _z\leq d-2$, hence $f\leq d+4$ and $e\leq 2d-4$.
In the remaining case, 
$p'_0\geq 2$ and $p''_0\geq 2$. This means that $P$ contains a fat point of multiplicity $2$
in the tangent plane at the vertex, so $\mu _0=2$ and $\sum n_z\geq 2$. Therefore
as before $\sum \delta _z\leq d-2$, hence $f\leq d+4$ and $e\leq 2d-4$.
This completes the treatment of the last case (b9). 
\end{proof}

In the next sections we apply these considerations to the cases $d\leq 9$. 

\subsection{When $d=9$}

\begin{proposition}
\label{d9cor}
For a general quintic, the irreducible component of $M_X(2,-1,9)$ considered
in Theorem \ref{only9}, is the only one; this completes the proof of that theorem.
\end{proposition}
\begin{proof}
The expected dimension is $16$, so every component of the moduli space has
dimension $\geq 16$. 
In Proposition  \ref{iclassif}, case (a) is ruled out by $d=9>8$ and cases (e) and (f)
are ruled out because we need $c<9$. By Lemmas \ref{dimcount1}  and \ref{inYdims},
the dimensions of the strata for cases (b) and (c) are $\leq 14$. 

Consider the strata corresponding to Case (d). Note that $c=8$ so
the number of choices of extension $d-c-1$ vanishes; hence $e$ is equal to the
dimension of the space of choices of $P$. And of course 
the dimension of the image of $\Sigma _d$ in the moduli space of bundles
is $\leq e$.

Refer to the cases of Lemma \ref{dcases}.
In all cases except (d1) and (d3), the curve $Y$ is contained in a union of two planes.
For these cases, consider the variety of triples $(UU',P)$ where $H=UU'$ is a union of
two planes, and $P\subset H\cap X$. The family of choices of $UU'$ has dimension $6$,
and for each choice, the space of choices of $P$ has dimension $\leq d$ by the
BGS estimate \cite{BGS}. Hence, the family of choices of $P$ in all cases except (d1) and (d3),
has dimension $\leq d+6=15$. These strata cannot therefore give any additional irreducible
components. 

For case (d1), once $Y$ is fixed, the choice of $P$ is discrete unless $X$ is tangent
to one of the ordinary double points. Consider the stratum of choices such that $X$ is 
tangent to $a$ double points. The dimension of the space of choices of $Y$ is $\leq 16-3a$
and for given $Y$ the dimension of the space of $P$ is $\leq a$; hence the dimension of the
space of choices of $P$ is $\leq 14$ and, again, this can't contribute any new irreducible
components. 

For case (d3), proceed similarly: if $X$ is not tangent to the containing surface at
the cusp, then the choice of $P$ is discrete; otherwise, the dimension of the family of $Y$ is $\leq 3$ but $Q=X\cap Y$ is contained in
a cuspidal curve and the space of choices of $P_z$ at the cusp $z$, again has dimension 
$\leq 1$. This can be shown by an argument using the normalization of the cuspidal
curve and the conductor ideal. So the total dimension of the space of choices of $P$ is
again $\leq 14$ and this case doesn't contribute any new irreducible components.
This finishes the list of cases to be treated, proving that the moduli space is
irreducible.
\end{proof}

\subsection{When $d=8$}
$\mbox{ }$
\newline
Recall the explicit constructions of families of bundles for $d=8$, in \S \ref{sec6}. 
Note that $c\leq 7$ by CB so we are in cases (a) or (b) of Proposition \ref{iclassif}.

Look first at Case (a) where, in view of $d=8$, $P_V$
has to be equal to the full complete intersection of three basis elements of $V$.
Then $V$ is equal to the space of all quadrics vanishing on $P_V$, an
application of Lasker's theorem \cite[p. 314]{EisenbudGreenHarris}. 
In particular, $c=7$ and the corresponding stratum in the moduli space of bundles is
equal to the space of subschemes $P$ in question. 

Consider the following incidence variety suggested
by A. Hirschowitz. Let $A$ denote the open subset of $\Grass (3,10)$ of $3$-planes
$V\subset \cc ^10= H^0(\Oo _{\pp ^3}(2))$ which define a zero-dimensional intersection $P_V$
of three quadrics. Note that $\dim (A)=21$. Let 
$$
I\subset A\times \pp H^0(\Oo _{\pp ^3}(5))
$$
denote the incidence variety of pairs $(V,X)$ such that $P_V\subset X$.
For each $V\in A$, the subscheme $P_V$ imposes $8$ conditions on quintic hypersurfaces,
indeed it already imposes $7$ conditions on quadrics, and it isn't too hard to
see, using the fact that not too many points can be coplanar or colinear, that
there exists a quintic passing through seven points but not the eighth, schematically.
Thus, the map $I\rightarrow A$ is smooth with fibers of codimension $8$  in
$\pp H^0(\Oo _{\pp ^3}(5))$. It follows that $I$ is a smooth codimension $8$ subvariety
of $A\times \pp H^0(\Oo _{\pp ^3}(5))$. The intersection of $I$ with the general fiber
of the projection $p_2$ to $\pp H^0(\Oo _{\pp ^3}(5))$ is therefore smooth and
has codimension $8$ in
$A$. Furthermore, if $A'\subset A$ is any open set, then the general fiber 
of $p_2$ on $A'$, is dense in the general fiber of $p_2$ on $A$. We conclude that if
$X$ is a general quintic hypersurface, then the space of extensions of type (a)
is a smooth variety of dimension $13$, containing as a dense open subset the 
space of extensions of type (a) with $V$ in any given open subset of the 
grassmanian. 

\begin{lemma}
\label{Ma}
For $d=8$, 
the space of extensions of type (a) is an irreducible smooth variety of dimension $13$.
\end{lemma}
\begin{proof}
The argument given above shows that it is smooth of dimension $13$, it remains to show
irreducibility. The above argument also shows the following: let $\eta$ denote the
generic point of the space of quintics, and $\overline{\eta}$ the generic geometric point
corresponding to a quintic $X_{\overline{\eta}}$. The Galois group ${\rm Gal} (\overline{\eta}/\eta )$ acts on the set of irreducible components of $M_{X_{\overline{\eta}}}(2,-1,8)$.
A more geometric vision of this action is to say that as $X$ moves around in an open
subset of the space of quintics, the fundamental group acts by permutation on the
set of components. In either point of view, irreducibility of the incidence variety
$I$ implies that the action is transitive. 

Go back to the example of $8$ points on a rational normal curve; recall that this defined a
$12$-dimensional family inside the open subset 
$M^{{\rm (a)}}_{X_{\overline{\eta}}}(2,-1,8)$ of the moduli space corresponding
to bundles of type (a).

A monodromy
argument shows that as the rational normal curve $Y$ moves around in its parameter space, 
the monodromy group acts transitively on the choice of subset
of $8$ out of the $15$ intersection points $X\cap Y$. There is an open subset of 
the space of rational normal curves parametrizing those whose intersection with
$X$ is a discrete set of $15$ distinct points. For a basepoint $Y_0$ the fundamental
group $\pi _1$ acts on $Y_0\cap X$. Since $6$ points determine a rational normal curve,
the action of $\pi _1$ can send the first six points to any other set of six, in 
any order. On the other hand, choose $5$ points and degenerate them to a configuration
of the form $X\cap D$ where $D$ is a general line. Then $Y$ degenerates to $D\cup Z$
where $Z$ is a conic meeting $D$ transversally, with $Z\subset U$ for $U$ a plane
transverse to $D$.  Then we can choose any $4$ points in $U$, which together with the fifth
point $U\cap D$ determine a conic $Z$ meeting $D$. The $4$ points can be moved around 
inside $U\cap X$ in an arbitrary way. From this construction it follows that 
the subgroup of $\pi _1$ fixing $5$ of the points, acts transitively on the set of
ordered quadruples of points in the remaining $10$. Therefore, any $8$ points can be moved to
any $8$ others.

The preceding paragraph shows that the $12$-dimensional family constructed in \S \ref{sec-rationalnormal} is  irreducible. However,
as we noted there, the dimension of the Zariski tangent space at a general point 
is $13$, and by Corollary \ref{dimsigmageq} (as in the proof of \ref{geq13}) it lives inside a $13$-dimensional component of the moduli space;
hence that component is smooth at a general point of our $12$-dimensional family,
and this defines canonically a single irreducible component of  
$M^{{\rm (a)}}_{X_{\overline{\eta}}}(2,-1,8)$. Being canonically defined, this component is
preserved by the action of 
${\rm Gal} (\overline{\eta}/\eta )$; but transitivity of that action,
implies that it is the unique irreducible component. This shows that 
$M^{{\rm (a)}}_{X_{\overline{\eta}}}(2,-1,8)$ is irreducible.
The same therefore holds for 
all quintic surfaces in an open subset of the parameter space. 
\end{proof}

\begin{theorem}
\label{d8irred}
On a general quintic for $d=8$, 
the moduli space $M_X(2,-1,8)$ is irreducible of dimension $13$, one more than
the expected dimension. It contains the
smooth subset of Lemma \ref{Ma} as an open dense subset.
\end{theorem}
\begin{proof}
From Lemma \ref{inYdims}, the dimensions of the strata corresponding to case (b) 
are all $\leq 12$, but any irreducible component has dimension $\geq 13$ by Corollary \ref{dimsigmageq}.
Note that cases (c)-(f) are ruled out by the fact that $c\leq 7$ for a
Cayley-Bacharach subscheme of degree $8$. Therefore, every irreducible component meets
the subset of bundles of type (a); and this is irreducible by Lemma \ref{Ma}.
\end{proof}

\subsection{When $d=6,7$}
$\mbox{ }$
\newline
As pointed out in \S \ref{sec-euler}, for $d\leq 7$ we have $h^0(\Jj _P(1))>0$. For 
$d=6,7$ the subscheme $P$ is too big to be contained in $D\cap X$ for a line $D$,
so it must be contained in a unique plane $U$. The conditions imposed by $C$ on 
quadrics factor through the $6$-dimensional space of conics on $U$, so $c\leq 6$.

Suppose $c\leq 4$. Then $P$ would be contained in two different conics on $U$;
but this is impossible. Indeed, if the conics intersect transversally in a 
scheme of length $4$ there isn't enough room for $P$, but
if the conics intersect in a line plus a point, the Cayley-Bacharach condition 
of $P$ rules out that $P$ could contain the extra point, so $P$ would be contained
in a line but again there isn't enough room. We conclude that $c=5$ or $c=6$.

\begin{proposition}
For a general quintic surface $X$ the moduli space $M_X(2,-1,7)$ has
a single irreducible component of dimension $9$ whose general point
corresponds to an extension where $P$ is a 
general arrangement of $7$ points on a curve of the form $U\cap X$
for a general plane $U$. Each bundle corresponds to a one-dimensional space of choices of
$(U,P)$. The moduli space is generically non-reduced, with Zariski tangent space
of dimension $11$.
\end{proposition}
\begin{proof}
For every bundle $E$, we have $h^0(E(1))=2$. Indeed by the Euler characteristic 
$h^0(E(1))\geq 2$ (see \S \ref{sec-euler}) 
but a subscheme $P$ of length $7$ cannot be contained in
a line, so $h^0(\Jj _P(1))\leq 1$ which gives $h^0(E(1))\leq 2$. 

Consider the variety of extensions $\Sigma _7$. This maps to the moduli space of bundles,
with fibers of dimension $1$: the fibers correspond to the choice of a line
inside the $2$-dimensional space $H^0(E(1))$. So, it suffices to show that $\Sigma _7$
has a single  irreducible component of dimension $10$. 

By Corollary \ref{dimsigmageq}, 
every irreducible component of $\Sigma _7$ has dimension $\geq 3d-11=10$.
For each extension, the plane $U$ containing $P$ is unique. If we fix $U$ then 
the space of choices of $P$ is the Hilbert scheme of a plane curve, so it has dimension 
$d=7$. A general choice of $P$ yields $c=6$ and we get a family of dimension $10$.
This family is irreducible, because for a general $U$ the curve $X\cap U$ is smooth
and the Hilbert scheme is irreducible. The subvariety of choices corresponding to 
$c=6$, where $U\cap X$
is singular, has dimension $\leq 9$ so it cannot form an irreducible component of $\Sigma _7^6$.  

Suppose $c=5$. Then $P$ is contained in a conic $Y\subset U$. 

For strata where
the singularities of $Y$ don't meet the singularities of $X\cap U$, the 
choice of $P\subset Y\cap X$ is discrete, and the dimension of the family of 
choices of $P$ is $3$ for the plane $U$,  plus $5$ for the conic. This gives $f=8$
and $e=f+d-c-1 = 9$. The dimension of these strata of $\Sigma ^5_7$ is therefore $\leq 9$
and they cannot contribute new irreducible components.

The space of choices $(U,Y,P)$ such that $U\cap X$ and $Y$ share a common double point,
has dimension bounded by $2$ for the plane $U$ tangent to $X$, plus $2$ for the
singular conic with fixed vertex, plus $1$ for a choice of $P\subset Q$ since $Q$ is
contained in a planar curve with a single node. Adding $d-c-1=1$ we get $e\leq 6$
so this case doesn't give any new components.

The space of choices of $(U,P)$ such that $U$ has worse than a single double point,
has dimension $1$ for the choice of $U$, plus at most $7$ for the choice of $P$
in the curve $U\cap X$ (this is undoubtedly not the best possible bound but it suffices for
our purposes). Adding $d-c-1=1$, we conclude that these strata of $\Sigma ^5_7$ have
dimension $\leq 9$ so they don't add any new components.

Altogether, $\Sigma ^5_7$ has dimension $\leq 9$ so the only irreducible component of $\Sigma _7$ is the general one of dimension $10$ constructed at the start. The moduli space
is irreducible of dimension $9$. 

The expected dimension is $8$ and the space of obstructions has dimension $3$ by
Lemma \ref{planesit}, so the Zariski tangent space has dimension $11$ at a general point;
therefore the  moduli space is generically non-reduced. 
\end{proof}

\begin{proposition}
For a general quintic surface $X$ the moduli space $M_X(2,-1,6)$ has
a single irreducible component of dimension $7$ whose general point
corresponds to a choice of $6$ points out of the $10$
in $X\cap Y$ for $Y$ a planar conic. The moduli space is generically smooth. 
\end{proposition}
\begin{proof}
When $d=6$,
the condition $c<d$ and the argument at the start of this section show that $c=5$.
As in the previous proposition, for any $E$ in the moduli space,
$H^0(E(1))$ has dimension $2$ so the space of extensions $\Sigma ^5_6$
fibers over the moduli space of bundles with fibers $\pp ^1$.

So, look at the space $\Sigma ^5_6$ of extensions.
Since $d-c-1=0$, for any $P$ the extension class is uniquely determined up to scalars,
and $\Sigma ^5_6$ is just the Hilbert scheme of appropriate $P$. As previously,
any $P$ is contained in a unique plane $U$ and, since $c=5$, it is contained in a unique
conic $Y\subset P$. By Corollary \ref{dimsigmageq}, 
every irreducible component of $\Sigma ^5_6$
has dimension $\geq 3d-11 = 7$. 

Write $\Sigma ^5_6=\Sigma \cup \Sigma ' \cup \Sigma ''$ where 
$\Sigma$ is the locus where $Y$ is a smooth conic, $\Sigma ''$ is the locus where
$Y$ is a union of two distinct lines, and $\Sigma ''$ is the locus where $Y$ is
a double line. 

Given $(U,Y)$ with $Y$ a smooth conic, then the space of choices of $P\subset Q:=Y\cap X$
is discrete since $Y$ is a smooth curve. By a monodromy argument, the
covering of the space of $(U,Y)$ corresponding to choosing $6$ points 
in the intersection $Q$ of length $10$, is irreducible. 
The argument is analogous to the one used in the proof of Lemma \ref{Ma}.
The fundamental group $\pi _1$ of the space of conics which intersect $X$ transversally,
acts on the intersection set of $10$ points. Since $5$ points determine a conic, we can
position the first $5$ points arbitrarily. On the other hand, degenerating $5$ points to
a line, then a sixth  point can be positioned arbitrarily, together with a $7$th point
they determine the other line in a reducible conic. 

Thus,
the closure of $\Sigma$ is an $8$-dimensional
irreducible component of $\Sigma ^5_6$. 

Consider the map from $\Sigma '$ to the space $\Ff$
of $(U,Y)$ where $Y$ is a union of two lines.
Over the open set $\Ff ^0$ where the vertex of $Y$ is not a point of tangency of $U$ to $X$,
the intersection $Q:=Y\cap X$ is a subscheme of a smooth curve at every point
(either the curve $Y$ or the curve $X\cap U$); so
the fiber (i.e. the space of choices of $P$) is discrete over such points. 
The space $\Ff ^0$ has dimension $7$, and the corresponding piece of $\Sigma '$ is a
covering corresponding to the choice of $3$ points in each of the two
lines intersected with $X$. This $7$-dimensional family is not a new irreducible
component, indeed at a general point $Y$ can be smoothed and $P$ follows to 
generalize to a point of $\Sigma$, and special points would constitute strata of
dimensions $\leq 6$. 

Consider the complement $\Ff ^1= \Ff -\Ff ^0$ where the vertex of the two lines is
a point of tangency of $U$ to $X$. This has dimension $4$, i.e.
$2$ for choice of a tangent plane $U$,
plus $2$ for a choice of two lines passing through the point of tangency.
The space of choices of $P$ has dimension $\leq 1$, so the piece of $\Sigma '$
corresponding to $\Ff ^1$ has dimension $\leq 5$ and cannot furnish a new irreducible
component. We conclude from these two paragraphs that $\Sigma '$ is in the
closure of $\Sigma$.

For $\Sigma ''$, let $\Gg$ be the space of choices of $(U,Y)$ where $Y$ is a double
line in $U$. It has dimension $5$. Let $\Gg ^0$ be the subset consisting of 
pairs $(U,Y)$ such that $U\cap X$ has at most one double points on $Y$.
For any $(U,Y)\in \Gg _0$, the subscheme $Q=X\cap Y$ is a subscheme of a smooth
curve or, at most at one point an ordinary double point. So the space of choices of
$P$ has dimension $\leq 1$ over any point of $\Gg ^0$ and the dimension of this part of 
$\Sigma ''$ is $\leq 6$; hence it cannot contribute an irreducible component. 

Let $\Gg ^1$ be the space of choices of $(U,Y)$ such that $U\cap X$ has an ordinary cusp
at a point of $Y$. It has dimension $2$, equal to $1$ for the choice of $U$ plus $1$
for the choice of line passing through the cusp. For any $(U,Y)\in \Gg ^1$ the
subscheme $Q=Y\cap X$ is contained in the cuspidal curve $U\cap X$,
so the space of choices of $P$ has dimension $\leq 1$ (see the argument recalled in 
\S \ref{appendix} below). Hence this part of $\Sigma ''$ has dimension $\leq 3$. 

Let $\Gg ^2$ be the space of choices of $(U,Y)$ such that $U\cap X$ has two distinct
ordinary double points along $Y$. This has dimension $1$ since the line passing through
the double points is determined. The space of choices of $P$ has dimension $\leq 2$
so the dimension of this part of $\Sigma ''$ is $\leq 3$. 

Consider, last but not least, the complement $\Gg ^3$ of the subsets considered previously.
A point $(U,Y)$ consists of a plane with higher order of contact than an ordinary cusp,
or else a cusp and a double point. The space of choices of $U$ has dimension $0$.
At this point we should ignore the choice of $Y$ and point out that for
each $U$, the space of choices of $P\subset X\cap U$ has dimension $\leq 6$,
so the space of choices of $P$ leading to an element of $\Gg ^3$ cannot have dimension
bigger than $6$, so this part of $\Sigma ''$ cannot contribute a new irreducible
component. 

This concludes the proof that the only irreducible component of $\Sigma ^5_6$
is the $8$-dimensional one coming from a choice of $P$ in a general intersection
$X\cap Y$ for $Y$ a planar conic. As pointed out at the start and similarly to 
the preceding proposition, the dimension of the moduli space is one less: 
$M_X(2,-1,6)$ is irreducible of dimension $7$. 

The expected dimension is $4$ and by Lemma \ref{planesit} which applies everywhere,
the space of obstructions has dimension $3$ so the Zariski tangent space has dimension
$7$. This is the same as the dimension, so the moduli space is smooth---the obstructions
vanish. 
\end{proof}

\subsection{When $d=4,5$}

\begin{lemma}
\label{d45}
For $d=4,5$ the moduli space is irreducible of dimension $d-2$. 
\end{lemma}
\begin{proof}
Recall from \S \ref{sec-euler} that for $d=4,5$ the subscheme $P$ is 
contained in $Q=X\cap D$ for a line $D$. As $Q$ is contained in the smooth curve $D$,
the space of choices for $P$ is discrete. We obtain a map 
$\Sigma _d\rightarrow \Grass (1,\pp ^3)$ to the $4$-dimensional Grassmanian of 
lines in $\pp ^3$. For any subset $P$ of
length $\geq 4$, the number of conditions imposed on quadrics is $c=3$. 

For $d=5$, $\Sigma _d$ fibers over the 
Grassmanian of lines with fiber of dimension $1$, in fact the fibration is
the open subset of the $\pp ^1$-bundle of extension classes, given by
the condition of nonvanishing at the points of $P$ (see the discussion above
Corollary \ref{dimsigmageq}). 

For $d=4$
the space of collections of $4$ aligned points maps 
to the Grassmanian of lines by a $5$-fold ramified covering;  but the covering is still irreducible
(as can be seen by a monodromy argument upon moving the line). The extension class is unique
so the moduli space is isomorphic to this covering. 

In both cases, $\dim (\Sigma _d)=d$. 
However, the extension of the form \eqref{iseq} is
not uniquely determined by the bundle $E$. Indeed, $h^0(\Jj _P(1))=2$
so $h^0(E(1))=3$. An extension corresponds to a choice of line
in the three dimensional space $H^0(E(1))$. Therefore $\Sigma _d$ is a $\pp ^2$-bundle over $M_X(2,-1,d)$
which gives the stated dimension. 
\end{proof}

It should be interesting to study the spectral varieties of co-obstructions here.

\subsection{Appendix: techniques for estimating the dimensions}
\label{appendix}

In this appendix we state more explicitly the bounds on the dimensions of (local)
Hilbert schemes that we are using. 
The general results of  \cite{BGS} bound the dimension of the global
Hilbert scheme of subschemes of a plane curve, and give some bounds for the 
dimension of the local Hilbert schemes. 

\begin{proposition}[Brian\c{c}on, Granger, Speder \cite{BGS}]
If $Z\subset X$ is a curve in a smooth surface $X$, then the Hilbert scheme of
subschemes $P\subset Z$ of length $d$ has pure dimension $d$. If $z\in Z$
is a point, then the Hilbert scheme of subschemes $P\subset Z$ of length $d$ 
supported set-theoretically at $z$, has dimension $\leq d-1$. Furthermore,
the stratum of subschemes $P$ containing the fat punctual
subscheme of $X$ defined by $\mathfrak{m}_z ^{\nu}$, has dimension $\leq d-\nu$. 
\end{proposition}

Of course, if $Z$ is a smooth curve then any subscheme is locally defined by a
power of a uniformizing parameter, so the local Hilbert scheme has dimension $0$.
It is useful to know about the next two cases, 
when $Z$ has an ordinary double point or a cusp. In these cases, 
the local dimensions are $\leq 1$. These examples are implicit in
the discussion of \cite{BGS} and should be assumed as well-known, 
however the cusp case doesn't seem to
have been mentioned explicitly and we haven't yet found a good reference. For convenience we give an argument.

\begin{lemma}
Suppose $Z$ has an ordinary double point or a 
cusp at $z$. The dimension $\delta _z$ of the Hilbert scheme
of subschemes $P\subset Z$ of length $d$, 
set-theoretically supported at $z$, is $\delta _z\leq 1$.
\end{lemma}
\begin{proof}
Choose first subschemes 
$R_z\subset T_z\subset Z$ (supported at $z$) and then look at $P_z$ lying between $R_z$ and $T_z$. The idea is to construct $R_z$ and $T_z$ depending on $P_z$, but in
such a way that they vary in a much smaller and in fact discrete family.
The dimension $\delta _z$ will be
bounded by the 
dimension of the space of choices of $R_z$ and $T_z$, plus the dimension, denoted $\delta _z(R_z,T_z)$, of the space $S(R_z,T_z; d)$ 
of choices of $R_z\subset P_z \subset T_z$. Using the inclusion
of $S(R_z,T_z; d)$ in the Grassmanian of quotients of $\Jj _{R_z/T_z}$ of length $\ell  (\Jj _{R_z/P_z})$ we get
\begin{equation}
\label{grassest}
\delta _z(R_z,T_z)\leq \ell (\Jj _{R_z/P_z})\ell (\Jj _{P_z/T_z}).
\end{equation}

Let $\tilde{Z}$ be the normalization of the curve $Z$.
Let $\tilde{R}_z\subset \tilde{Z}$ be the subscheme of $\tilde{Z}$ defined by the pullback
of the equations defining $P_z$; and let $R_z\subset Z$ be the smallest subscheme containing
the image of $\tilde{R}_z$ scheme-theoretically, in other words $\Jj _{R_z/Z}$ is generated by
all the functions which, when pulled back to $\tilde{Z}_z$, vanish on $\tilde{R}_z$.

The data of $\tilde{R}_z\subset Z$ is determined by the lengths
of the subschemes at each point upstairs in the normalization,
since there we are in a smooth curve. 
Thus, the set of possible
choices of $\tilde{R}_z$ and hence of $R_z\subset Z$, has dimension zero.
We may assume that $\tilde{R}_z$ and $R_z$ are fixed.  
Notice on the other hand, by construction, that $R_z\subset P_z$. 

To construct the outer subscheme $T_z$, use the {\em conductor ideal} 
${\bf a}\subset \Oo _{\tilde{Z}}$  of the extension
$\Oo _Z\hookrightarrow \Oo _{\tilde{Z}}$. 
That is the ideal ${\bf a}$ of functions $u$ such that for any 
$v\in \Oo _{\tilde{Z}}$, the product $uv$ is in $\Oo _Z$.  Applying the definition for $v=1$
we conclude that ${\bf a}\subset \Oo _Z$ too.
On the other hand, we have
the ideal $\Jj _{\tilde{R}_z/\tilde{Z}}$ which is generated by the generators of 
$\Jj _{P_z/Z}$, in other words $\Oo _{\tilde{R}_z}=\Oo _{\tilde{Z}}\otimes _{\Oo _Z}\Oo _P$.
Furthermore, 
$$
\Jj _{R_z/Z}= \Jj _{\tilde{R}_Z/\tilde{Z}}\cap \Oo _Z.
$$
Since clearly $\Jj _{P_z/Z}\subset \Jj _{\tilde{R}_z/\tilde{Z}}$, we also have
$$
\Jj _{P_z/Z}\subset \Jj _{R_z/Z}. 
$$
On the other hand, by definition of ${\bf a}$ we have
${\bf a}\Jj _{\tilde{R}_z/\tilde{Z}}\subset \Oo _Z$, so this is the ideal of a subscheme 
$$
\Jj _{T_z/Z}:= {\bf a}\Jj _{\tilde{R}_z/\tilde{Z}}.
$$
Now $\Jj _{\tilde{R}_z/\tilde{Z}}$ is the image of the map 
$$
\Oo _{\tilde{Z}}\otimes _{\Oo _Z}\Jj _{P_z/Z}\rightarrow \Oo _{\tilde{Z}}.
$$
Therefore $\Jj _{T_z/Z}$ is the image of the map  
$$
{\bf a}\otimes _{\Oo _Z}\Jj _{P_z/Z}\rightarrow \Oo _Z\subset \Oo _{\tilde{Z}}
$$
and we have $\Jj _{T_z/Z}\subset \Jj _{P_z/Z}$. 

Altogether, once $\tilde{R}_z$ is fixed, we get the subschemes $R_z\subset T_z$
and $P_z$ lies between the two. In the present case, the dimension of the
space of choices of $R_z$ and $T_z$ is zero, so we get from \eqref{grassest} 
\begin{equation}
\label{normalest}
\delta _z \leq \ell (\Jj _{R_z/P_z})\ell (\Jj _{P_z/T_z}).
\end{equation}

For an ordinary double point, the conductor ideal ${\bf a}$ is
just the maximal ideal at $z$ and we have $\Jj _{R_z/T_z}\cong \Jj _{R_z/Y}\otimes _{\Oo _Y}(\Oo _Y/{\bf a})$. Its length is
the number of generators of 
$\Jj _{R_z/Z}$ which is always $2$, thus
$$
\ell (\Jj _{R_z/T_z}) = \ell (\Jj _{R_z/P_z})+\ell (\Jj _{P_z/T_z}) =2.
$$
Hence $\delta _z\leq 1$ by \eqref{normalest}. 

At a cusp, the coordinate on the normalized curve is denoted $y$,
and the conductor ideal ${\bf a}$ is $(y^2)$. This is also the maximal ideal of $\Oo _Z$.
Just as for a node, the ideal $\Jj _{R_z/Y}$ has at most two generators as can be 
seen by running explicitly through the possibilities for $\tilde{R}_z$. 
Therefore $\ell (\Jj _{R_z/P_z})+\ell (\Jj _{P_z/T_z})\leq 2$ and $\delta _z\leq 1$ by \eqref{normalest}. 
\end{proof}

\bibliographystyle{amsplain}

\end{document}